\documentclass{amsart}
\usepackage[utf8]{inputenc}
\usepackage[T1]{fontenc}
\usepackage{amsmath,amsthm,amsfonts,amssymb}
\usepackage{microtype}
\microtypesetup{expansion=false}
\usepackage{listings}
\usepackage{xcolor}
\usepackage[hidelinks]{hyperref}
\newtheorem{theorem}{Theorem}[section]
\newtheorem{lemma}[theorem]{Lemma}
\newtheorem{proposition}[theorem]{Proposition}
\newtheorem{corollary}[theorem]{Corollary}
\theoremstyle{remark}

\numberwithin{equation}{section}
\definecolor{codebg}{RGB}{248,248,248}
\definecolor{codeframe}{RGB}{220,220,220}
\definecolor{codenumbers}{RGB}{120,120,120}
\definecolor{wlkeyword}{RGB}{0,92,184}
\definecolor{wlstring}{RGB}{163,21,21}
\definecolor{wlcomment}{RGB}{90,90,90}
\lstdefinelanguage{Wolfram}{
  morekeywords={AllTrue,AnyTrue,Binomial,Cancel,Catch,Coefficient,
    CoefficientList,CoefficientRules,Complement,Count,D,Do,Expand,
    Exponent,First,Flatten,Function,Head,If,IntegerQ,Join,Last,Length,
    List,Min,Module,Ordering,PolynomialQ,Print,Product,Rational,Rest,
    Select,Sum,Table,Throw,Together,Total,True,TrueQ,Variables,With},
  sensitive=true,morecomment=[s]{(*}{*)},morestring=[b]"
}
\lstdefinestyle{wlcode}{
  language=Wolfram,basicstyle=\ttfamily\fontsize{9}{10.5}\selectfont,
  backgroundcolor=\color{codebg},frame=single,
  rulecolor=\color{codeframe},framerule=0.4pt,
  columns=fullflexible,keepspaces=true,showstringspaces=false,
  breaklines=true,breakatwhitespace=false,tabsize=2,upquote=true,
  numbers=left,numberstyle=\scriptsize\color{codenumbers},numbersep=8pt,
  xleftmargin=2.2em,framexleftmargin=1.6em,
  keywordstyle=\color{wlkeyword}\bfseries,
  commentstyle=\color{wlcomment}\itshape,
  stringstyle=\color{wlstring}
}
\begin{document}
\title[Askey's convexity conjecture]
{A proof of Askey's convexity conjecture}
\author{K. Castillo}
\address{CMUC, Department of Mathematics, University of Coimbra,
3000-143 Coimbra, Portugal}
\email{math@keniercastillo.com}
\author{S. Yakubovich}
\address{Department of Mathematics, Faculty of Sciences,
University of Porto, Campo Alegre st., 687, 4169-007 Porto, Portugal}
\email{syakubov@fc.up.pt}
\date{September 9, 2026}
\begin{abstract}
For $-1<\alpha\leq1/2$, let $J_\alpha$ be the Bessel function of the
first kind and let $j_{\alpha,2}$ be its second positive zero.
Define $\beta(\alpha)<\alpha+1$ by
$$
\int_0^{j_{\alpha,2}}u^{-\beta(\alpha)}J_\alpha(u)\,du=0.
$$
We prove that $\beta''(\alpha)>0$ for $-1<\alpha\leq1/2$, including
the left second derivative at $\alpha=1/2$. The continuous extension
$\beta(-1)=0$ is strictly convex on $[-1,1/2]$, strengthening Askey's
1993 convexity conjecture. The analytic argument uses a positive
expansion of a primitive and a vanishing weighted sum. Increasing
ratios of consecutive weights give a negative covariance term, while
one comparison controls slope and curvature. The remaining algebraic
step proves eight rational inequalities by finite exact polynomial
calculations, reproduced in an appendix.
\end{abstract}
\maketitle

\section{Introduction}\label{sec:introduction}

In his 1993 article \emph{Problems which interest and/or annoy me},
Askey \cite[p.~8]{Askey1993} conjectured the convexity of the boundary
of a positivity region for Bessel integrals. Let $J_\alpha$ be the
Bessel function of the first kind and $j_{\alpha,2}$ its second
positive zero. Following Askey
\cite[p.~7, (1.26)]{Askey1993}, define $\beta(\alpha)<\alpha+1$ by
\begin{equation}\label{eq:boundary-intro}
\int_0^{j_{\alpha,2}}u^{-\beta(\alpha)}J_\alpha(u)\,du=0,
\quad -1<\alpha\leq\frac12.
\end{equation}
Existence and uniqueness follow from the classical positivity theory
\cite{AskeySteinig1974,Makai1974}.
The boundary has the continuous extension $\beta(-1)=0$
\cite[p.~7, (1.29)]{Askey1993}. With this endpoint convention,
Askey's conjecture asserts that $\beta$ is convex on $[-1,1/2]$.

Using hypergeometric positivity, Cho, Chung, and Yun
\cite{ChoChungYun2020} obtained the following bounds:
$$
\max\{-\alpha-1,-1/2\}<\beta(\alpha)
\leq-\frac{\alpha+1}{3},
\quad -1<\alpha<\frac12,
$$
and $\beta(1/2)=-1/2$
\cite[Theorem~5.1 and Corollary~5.1]{ChoChungYun2020}.
Their 2026 correction~\cite{ChoChungYunCorrection2026} repairs
Case~II of the proof of Lemma~3.1 without changing its statement or
these positivity bounds.
Building on the positivity bounds of Cho, Chung, and Yun, we establish
strict convexity of the boundary, thereby resolving Askey's convexity
conjecture. More precisely, we prove
$$
\beta''(\alpha)>0,\quad -1<\alpha\leq\frac12.
$$
The proof also uses the real-zero theorem of Cho, Chung, and Park
\cite[Theorem~7.2(A)]{ChoChungPark2026}.

\begin{theorem}\label{thm:main}
For real $\alpha>-1$, let $J_\alpha$ be the Bessel function of the
first kind, defined by
$$
J_\alpha(x)=\sum_{k=0}^\infty
\frac{(-1)^k(x/2)^{2k+\alpha}}{k!\,\Gamma(k+\alpha+1)},
\quad x>0,
$$
where $\Gamma$ is Euler's gamma function. Denote its positive zeros,
in increasing order, by $j_{\alpha,1},j_{\alpha,2},\ldots$.
For real $b<\alpha+1$, put
$$
I_{\alpha,b}(x)=\int_0^x u^{-b}J_\alpha(u)\,du,
\quad x>0,\quad I_{\alpha,b}(0)=0,
$$
where the integral at zero is understood as an improper integral.
For every $-1<\alpha\leq1/2$, there is exactly one real number
$\beta(\alpha)<\alpha+1$ such that
$$
I_{\alpha,\beta(\alpha)}(j_{\alpha,2})=0.
$$
The resulting function $\beta$ is real analytic on $(-1,1/2)$,
extends real analytically to a neighbourhood of $1/2$, and satisfies
$$
\beta(1/2)=-\frac12,
\quad \beta''(\alpha)>0,
\quad -1<\alpha\leq\frac12.
$$
Here $\beta''(1/2)$ is the second derivative of the analytic
extension, which agrees with the left second derivative. The function
has a continuous extension to $[-1,1/2]$ defined by $\beta(-1)=0$,
and this extension is strictly convex.
\end{theorem}

The central point is to use the defining equation twice. After the
change of variables $t=\alpha+1$ and $\beta(t-1)=t-2Y(t)$,
a positive expansion gives $Y'$ as a weighted average together with
a vanishing moment. The first use of this cancellation pairs the
first two indices with the remaining terms and yields a common upper
bound for $Y'$. The second use applies the same pairing to the
curvature identity, where increasing weight ratios contribute a
strictly negative covariance. Monotonicity in an auxiliary real parameter then
allows every curvature term to be compared at the same upper bound.

Section~\ref{sec:contact} constructs the curve and the positive
weights. Section~\ref{sec:curvature} gives the curvature argument and
assembles the proof of Theorem~\ref{thm:main} from four lemmas.
The analytic estimates in Sections~\ref{sec:bounds}--\ref{sec:trace}
control the Bessel zero and the weights as the parameter varies, and
extend the required comparisons to every index. The finite algebraic
input is isolated in Proposition~\ref{lem:finite}: eight rational
inequalities in two independent variables. Section~\ref{sec:algebra}
reduces these inequalities, by monotonicity and concavity, to polynomial
identities and coefficient signs. Appendix~\ref{app:mathematica} gives
self-contained Mathematica code for those exact calculations in place
of lengthy coefficient tables. These checks are needed to complete
the proof. The code carries out only this finite algebraic step;
no numerical integration, parameter mesh, enclosure of
Bessel zeros, or floating-point sign decision is used.

Askey's convexity conjecture is distinct from the positivity problem
for finite Jacobi sums discussed in \cite{Askey1972,Gasper1977}.
Gasper \cite[p.~444]{Gasper1977} proposed the same function $\beta$
as the boundary of the remaining positivity region.
Theorem~\ref{thm:main} proves strict convexity of that boundary,
not the sufficiency assertion for the finite sums.

\section{The defining equation and a positive primitive}\label{sec:contact}

We use the rising factorial $(a)_0=1$ and
$(a)_n=a(a+1)\cdots(a+n-1)$ for $n\geq1$, and Euler's beta integral
$$
B(a,b)=\int_0^1u^{a-1}(1-u)^{b-1}\,du,
\quad a,b>0.
$$
For convergence under integrals we use the dominated convergence
theorem, differentiation under the integral, and the Fubini--Tonelli
theorem in \cite[Theorems~2.24, 2.27, and~2.37]{Folland1999}.
The analytic implicit-function theorem is used in the form of
\cite[Section~2.3]{KrantzParks2002}.

\begin{lemma}[Existence and regularity]\label{lem:contact}
For $0<t\leq3/2$, put $x(t)=j_{t-1,2}$ and $X(t)=x(t)^2/4$.
There is exactly one $Y(t)>0$ such that
\begin{equation}\label{eq:contact-bessel}
\int_0^{x(t)}u^{2Y(t)-t}J_{t-1}(u)\,du=0.
\end{equation}
Both $X$ and $Y$ are real analytic on $(0,3/2)$ and extend real
analytically across $3/2$. Moreover, with
$q(t)=t-Y(t)$, one has $q(t)>0$ and $Y'(t)>0$. At $t=3/2$,
$X(3/2)=\pi^2$ and $Y(3/2)=1$.
\end{lemma}

\begin{proof}
For $0<t<3/2$, existence is classical
\cite{AskeySteinig1974,Makai1974}; see
\cite[Theorem~B(i)]{ChoChungYun2020}, with
$\alpha=t-1$ and $\beta=t-2Y$.
At $t=3/2$, the identity $J_{1/2}(s)=\sqrt{2/(\pi s)}\sin s$
\cite[(10.16.1)]{DLMF}
gives $x=2\pi$ and the solution $Y=1$.
For $0\leq u\leq1$, set
\begin{equation}\label{eq:PQ}
Q(t,u)={}_0F_1(;t;-X(t)u),
\quad P(t,u)=\int_0^u v^{Y(t)-1}Q(t,v)\,dv.
\end{equation}
The substitution $s=x(t)\sqrt{u}$ in
\eqref{eq:contact-bessel} gives $P(t,1)=0$, while $Q(t,1)=0$.
For fixed $t$, let $u_0=(j_{t-1,1}/x(t))^2$ be the only zero
of $Q(t,u)$ in $(0,1)$, and let $Y_0>0$ satisfy
\eqref{eq:contact-bessel}.
For any $y>0$,
$$
\int_0^1u^{y-1}Q(t,u)\,du
=\int_0^1u^{Y_0-1}Q(t,u)
\{u^{y-Y_0}-u_0^{y-Y_0}\}\,du.
$$
The right side is strictly negative for $y>Y_0$ and strictly
positive for $y<Y_0$, by the change of sign of $Q$ at $u_0$.
This ensures uniqueness among all positive exponents, including
at $t=3/2$.
When $t$ is fixed, we suppress it in $P,Q$ and write
$Y=Y(t)$, $X=X(t)$, $q=t-Y$. Spatial derivatives are written as
$\partial P/\partial u$, whereas primes on $X,Y$ differentiate $t$.
The identity
$$
Q(t,u)=\Gamma(t)(\sqrt{Xu})^{1-t}
J_{t-1}(2\sqrt{Xu}),\quad 0<u\leq1,
$$
shows that $P$ first increases and then decreases to zero. Therefore
\begin{equation}\label{eq:positive-primitive}
P(t,u)>0,\quad 0<u<1,
\quad \frac{\partial Q}{\partial u}(t,1)>0.
\end{equation}
Also $P(t,u)=O(u^Y)$ at zero and $P(t,u)=O((1-u)^2)$ at one.

The positive Bessel zeros are simple \cite[Section~10.21(i)]{DLMF};
the analytic implicit-function theorem therefore makes $X(t)$
real analytic for $t>0$.
The function $M(t,y)=\int_0^1u^{y-1}Q(t,u)\,du$ is real analytic
for $t,y>0$, as follows from its normally convergent Bessel series.
At a solution, integration by parts gives
$$
\frac{\partial M}{\partial y}(t,Y(t))
=-\int_0^1\frac{P(t,u)}u\,du<0.
$$
The analytic implicit-function theorem proves the asserted regularity
of $Y$, including its extension across $3/2$.

The equation for $Q$ and two integrations by parts give
\begin{equation}\label{eq:trace-integrals}
\begin{aligned}
u\frac{\partial^2Q}{\partial u^2}
+t\frac{\partial Q}{\partial u}+XQ&=0,\\[7pt]
\frac{\partial Q}{\partial u}(t,1)
&=Xq\int_0^1u^{q-1}P(t,u)\,du
=X\int_0^1P(t,u)\,du.
\end{aligned}
\end{equation}
The first equality in the second line is legitimate even before the
sign of $q$ is known: $u^qP=O(u^t)$ and
$u^{q-1}P=O(u^{t-1})$. Its positive left side proves $q>0$.
For the last equality, multiply the differential equation by $u^Y$;
the term left after the integrations by parts is a multiple of
$\int_0^1u^{Y-1}Q\,du=0$.

For $w\in C^1[0,1]$, introduce
$$
(A_tw)(u)=\int_0^1\frac{v^{t-1}}{1-v}\{w(uv)-w(u)\}\,dv.
$$
Its value on $u^k$ is $-u^k\sum_{j=0}^{k-1}(t+j)^{-1}$.
Termwise differentiation of the series for $Q$ consequently gives
$$
\frac{\partial Q}{\partial t}(t,u)
=(A_tQ)(u)+\frac{X'(t)}{X(t)}u\frac{\partial Q}{\partial u}(t,u).
$$
The two endpoint identities imply
$$
\int_0^1u^Y\frac{\partial Q}{\partial u}\,du=0,
\quad \int_0^1u^{Y-1}\log u\,Q\,du
=-\int_0^1\frac{P(t,u)}u\,du.
$$
Fubini's theorem, followed by $s=uv$, gives
$$
\int_0^1u^{Y-1}(A_tQ)(u)\,du
=\int_0^1\frac{v^{q-1}P(t,v)}{1-v}\,dv.
$$
To justify the cancellation in this formula, first restrict the
$v$-integral to $v<1-\varepsilon$ and use the defining equation before passing to
the limit. The original difference is bounded in absolute value by
$\|\partial Q/\partial u\|_\infty(1-v)$, which provides an
integrable majorant. Differentiating $P(t,1)=0$ now yields
\begin{equation}\label{eq:first-variation}
Y'(t)=\frac{\displaystyle\int_0^1u^{q-1}P(t,u)/(1-u)\,du}
{\displaystyle\int_0^1P(t,u)/u\,du}>0.
\end{equation}
All parameter differentiations are dominated locally by integrable
powers of $u$ times finite powers of $|\log u|$.
The endpoint values were identified at the start of the proof.
\end{proof}

It follows that $0<Y\leq1$ for $0<t\leq3/2$. Moreover,
\begin{equation}\label{eq:curvature-conversion}
\beta(t-1)=t-2Y(t),\quad \beta''(t-1)=-2Y''(t).
\end{equation}
Thus it remains to prove $Y''(t)<0$ for $0<t\leq3/2$.

\begin{proposition}\label{lem:zero-separation}
For $0<t\leq3/2$ and $x=j_{t-1,2}$,
$x<j_{t+n-1,1}$ for every integer $n\geq3$.
Furthermore,
\begin{equation}\label{eq:positive-expansion}
u^{-Y}P(t,u)=\sum_{n=2}^\infty c_n(t)(1-u)^n,
\quad c_n(t)>0,\quad 0<u\leq1.
\end{equation}
The function on the left extends to an entire function of $u$.
\end{proposition}

\begin{proof}
Comparing the coefficients of degree four in the logarithms of the
product representation \cite[(10.21.15)]{DLMF} and the defining
series of $J_t$ gives
$$
\sum_{k=1}^\infty j_{t,k}^{-4}
=\frac1{16(t+1)^2(t+2)}.
$$
The positive zeros of $J_{t-1}$ and $J_t$ interlace
\cite[Section~10.21(i)]{DLMF}, so
$j_{t,1}<x<j_{t,2}$ and $J_t(x)<0$, while
$$
x^2>j_{t,1}^2>4(t+1)\sqrt{t+2}>4t(t+1).
$$
The recurrence \cite[(10.6.1)]{DLMF}, applied at the zero of
$J_{t-1}$, gives
$$
J_{t+2}(x)=\{4t(t+1)/x^2-1\}J_t(x)>0.
$$
Since $x<j_{t+2,2}$, the sign implies $x<j_{t+2,1}$.
The zeros increase with the order \cite[Section~10.21(iv)]{DLMF},
which proves the remaining separations.

The entire extension is
$H(t,u)=\int_0^1v^{Y-1}Q(t,uv)\,dv$.
Its coefficients in powers of $1-u$ are
\begin{equation}\label{eq:coeff-integral}
c_n(t)=\frac{X^n}{n!(t)_n}
\int_0^1v^{Y+n-1}{}_0F_1(;t+n;-Xv)\,dv.
\end{equation}
Contact gives $c_0=c_1=0$. The identity
$Q=YH+u\,\partial H/\partial u$ gives
$$
c_2=\frac12\frac{\partial Q}{\partial u}(t,1)>0.
$$
For $n\geq3$, zero separation makes the integrand in
\eqref{eq:coeff-integral} positive. Differentiating the same identity
twice and using the equation for $Q$ also gives
\begin{equation}\label{eq:first-coeff-ratio}
\frac{c_3}{c_2}=\frac{Y+t+2}{3}.
\end{equation}
\end{proof}

Put
\begin{equation}\label{eq:weights}
A_0(t)=\int_0^1\frac{P(t,u)}u\,du,
\quad w_n(t)=\frac{c_n(t)B(Y(t),n+1)}{A_0(t)},\quad n\geq2.
\end{equation}
Tonelli's theorem and \eqref{eq:positive-expansion} show that
$w_n>0$ and $\sum_{n\geq2}w_n=1$. Define, for independent $t,Y>0$,
\begin{equation}\label{eq:components}
\begin{aligned}
e_n(t,Y)&=\frac{(Y)_{n+1}}{(t)_{n+1}},
&u_n(Y)&=\frac{Y}{Y+n+1},\\[7pt]
v_n(t,Y)&=u_n(Y)-(t-Y)e_n(t,Y),
&g_n(t,Y)&=\frac{(Y)_{n+1}}{n(t)_n}.
\end{aligned}
\end{equation}
When arguments are omitted, these functions are evaluated at
$(t,Y(t))$. Beta integration in \eqref{eq:trace-integrals} and
\eqref{eq:first-variation} gives the two identities
\begin{equation}\label{eq:mean-trace}
Y'(t)=\sum_{n=2}^\infty w_ng_n,
\quad \sum_{n=2}^\infty w_nv_n=0.
\end{equation}
Also
\begin{equation}\label{eq:first-ratio}
\lambda(t,Y):=\frac{Y+t+2}{Y+3},
\quad \frac{w_3(t)}{w_2(t)}=\lambda(t,Y(t)).
\end{equation}

We record convergence needed later. On any compact subinterval of
$(0,3/2]$, the analytic functions $X,Y$ extend holomorphically to a
complex neighbourhood with $\operatorname{Re}t,
\operatorname{Re}Y>0$. The integral defining $H(t,u)$ is jointly
holomorphic there and entire in $u$. Applying Cauchy's estimates on
$|u-1|=R>1$ and on a smaller parameter neighbourhood gives
\begin{equation}\label{eq:normal-convergence}
\left|\frac{d^jc_n}{dt^j}\right|\leq C_RR^{-n},
\quad j=0,1,2.
\end{equation}
See \cite[Chapter~IV, Section~2, (2.14)]{Conway1978} for these estimates.
The beta factors and the rational factors in \eqref{eq:components},
together with their first two parameter derivatives, have at most
polynomial growth times powers of $\log n$. Thus every weighted
series differentiated below is locally absolutely and uniformly
convergent. In particular, $\sum|w_n'|<\infty$.
\section{The curvature argument}\label{sec:curvature}

We first give the argument that turns the two identities in
\eqref{eq:mean-trace} into strict curvature. The estimates used here
are proved in Sections~\ref{sec:bounds}--\ref{sec:algebra},
independently of the curvature conclusion.
Unless stated otherwise, $0<t\leq3/2$, $Y=Y(t)$, and $q=t-Y$.

The proof has three steps. Lemma~\ref{lem:curvature-reduction}
bounds $Y''(t)$ by a weighted average of the values $K_n(Y'(t))$,
with positive weights, and shows that each $K_n$ increases on the
required interval. Lemma~\ref{prop:slope-majorant} gives $Y'(t)<R$, and
Lemma~\ref{lem:negative-kernels} gives $K_n(R)<0$ for every $n\geq4$;
the functions $K_n$ and the bound $R$ are defined below.
These facts imply $Y''(t)<0$. Subsection~\ref{sec:main-proof}
combines them with Lemma~\ref{lem:contact} to prove the theorem.

\subsection{Combining the positive and negative terms}

Combine indices two and three into the averages
\begin{equation}\label{eq:block}
g_A=\frac{g_2+\lambda g_3}{1+\lambda},
\quad v_A=\frac{v_2+\lambda v_3}{1+\lambda},
\quad \lambda=\frac{Y+t+2}{Y+3}.
\end{equation}
As in \eqref{eq:components}, these are rational functions of the
independent variables $t,Y$, evaluated on the curve when the arguments
are omitted. The subscript $A$ records this two-term average.
Proposition~\ref{lem:trace-sign} and Corollary~\ref{cor:domain} give
$v_A>0$ and $v_n<0$ for $n\geq4$. Thus, with $W=w_2+w_3$,
the cancellation in \eqref{eq:mean-trace} reads
$$
Wv_A+\sum_{n=4}^\infty w_nv_n=0.
$$
It follows directly that
\begin{equation}\label{eq:trace-pairs}
\begin{aligned}
\theta_n&=w_n\frac{v_A-v_n}{v_A}>0,
&r_n&=\frac{-v_ng_A+v_Ag_n}{v_A-v_n},\\[7pt]
\sum_{n=4}^\infty\theta_n&=1,
&Y'(t)&=\sum_{n=4}^\infty\theta_nr_n.
\end{aligned}
\end{equation}
The number $r_n$ is the ordinate at $v=0$ of the segment joining
$(v_A,g_A)$ and $(v_n,g_n)$. Since $g_n$ decreases strictly with $n$,
$$
g_n<r_n<g_A<g_2,\quad n\geq4.
$$
The same pairing will be used for the second derivative.

\subsection{The negative covariance term}

In the following functions, $r$ is a real parameter and $t,Y$ are fixed. Set
\begin{equation}\label{eq:directional-components}
H_n(s)=\sum_{j=0}^n\frac1{s+j},\quad s>0,
\quad G_n(r)=g_n\{rH_n(Y)-H_{n-1}(t)\}.
\end{equation}
Thus $G_n(Y'(t))$ is the total derivative of $g_n(t,Y(t))$.
The domain bound $t\geq7Y/5$ of Corollary~\ref{cor:domain} gives
\begin{equation}\label{eq:g2-unit}
g_2\leq\frac{25(Y+1)(Y+2)}{14(7Y+5)}\leq\frac{25}{28}<1;
\end{equation}
the middle quotient increases with $Y$ on $(0,1]$.

\begin{lemma}[Curvature reduction]\label{lem:curvature-reduction}
For $0<t\leq3/2$, define along the curve
\begin{equation}\label{eq:A-block}
\begin{aligned}
D(r)&=\frac{r+1}{Y+t+2}-\frac r{Y+3},\\[7pt]
A(r)&=\frac{G_2(r)+\lambda G_3(r)-D(r)(g_2-r)}{1+\lambda}
-\frac{g_A-r}{16},\\[7pt]
K_n(r)&=\frac{-v_nA(r)+v_AG_n(r)}{v_A-v_n},\quad n\geq4.
\end{aligned}
\end{equation}
Here $D(Y'(t))=d\log\lambda(t,Y(t))/dt$. The weights $\theta_n$
in \eqref{eq:trace-pairs} are positive and sum to one. The series
$\sum_{n\geq4}\theta_nK_n(r)$ converges absolutely for $0<r<g_2$, and
\begin{equation}\label{eq:curvature-reduction}
Y''(t)\leq\sum_{n=4}^\infty\theta_nK_n(Y'(t)),
\end{equation}
and each $K_n$ is strictly increasing for $0<r<g_2$.
\end{lemma}

\begin{proof}
Put $W=w_2+w_3=(1+\lambda)w_2$. Write $\sigma_n=w_n'/w_n$ and
$\delta_n=\sigma_{n+1}-\sigma_n$. The covariance of two sequences
with respect to the weights $w_n$ is
$$
\operatorname{Cov}_w(f,h)=\sum_{n\geq2}w_nf_nh_n
-\left(\sum_{n\geq2}w_nf_n\right)
 \left(\sum_{n\geq2}w_nh_n\right),
$$
whenever these sums are absolutely convergent. Differentiating
\eqref{eq:mean-trace} gives
\begin{equation}\label{eq:curvature-identity}
Y''(t)=\sum_{n\geq2}w_nG_n(Y'(t))+\operatorname{Cov}_w(g,\sigma).
\end{equation}
Absolute convergence follows from \eqref{eq:normal-convergence}:
$g_n$ is bounded and $\sum w_n|\sigma_n|=\sum|w_n'|<\infty$.
The identity
$\theta_nK_n(r)=w_n\{G_n(r)-v_nA(r)/v_A\}$ and the same estimates
give absolute convergence for every $r$ with $0<r<g_2$.
For
$$
A_k=\sum_{n=2}^kw_n(g_n-Y'(t))
=\sum_{i=2}^k\sum_{j=k+1}^\infty w_iw_j(g_i-g_j)>0,
$$
expanding the definition of covariance gives
\begin{equation}\label{eq:covariance-sum}
\operatorname{Cov}_w(g,\sigma)
=-\sum_{i<j}w_iw_j(g_i-g_j)(\sigma_j-\sigma_i)
=-\sum_{k=2}^\infty\delta_kA_k.
\end{equation}
For the last equality use
$\sigma_j-\sigma_i=\sum_{k=i}^{j-1}\delta_k$ and interchange the
nonnegative sums using \cite[Theorem~2.37(a)]{Folland1999};
all summands after the minus sign are positive by
Proposition~\ref{lem:weight-ratios}. This argument avoids any boundary term
in an infinite summation by parts.

The bound $\delta_3>1/16$ retains a uniform part of
the negative covariance; all later terms have the same sign and
may be discarded in an upper bound. Retaining $k=2,3$ and using
$A_2=w_2(g_2-Y'(t))$ and $A_3=W(g_A-Y'(t))$ gives
$$
Y''(t)\leq WA(Y'(t))+\sum_{n=4}^\infty w_nG_n(Y'(t)).
$$
The cancellation in \eqref{eq:mean-trace} and \eqref{eq:trace-pairs}
turn the right side
into \eqref{eq:curvature-reduction}.

For monotonicity, let
$d_1=(Y+t+2)^{-1}-(Y+3)^{-1}$. Direct differentiation gives
$$
A'(r)=\frac{g_2H_2(Y)+\lambda g_3H_3(Y)+D(r)-d_1(g_2-r)}{1+\lambda}
+\frac1{16}>0.
$$
Indeed, $D(r)>0$ on this interval. If $t\geq1$,
$d_1\leq0$; if $t<1$, then $0<d_1<H_2(Y)$ and
$g_2H_2(Y)-d_1(g_2-r)=g_2(H_2(Y)-d_1)+d_1r>0$.
Finally
$$
K_n'(r)=\frac{(-v_n)A'(r)+v_Ag_nH_n(Y)}{v_A-v_n}>0.
$$
\end{proof}

\subsection{One comparison for the slope and the curvature}

The form of a common upper bound follows from the pairing:
$$
r_n=g_A-v_A\frac{g_A-g_n}{v_A-v_n},\quad n\geq4.
$$
A positive common strict lower bound $\kappa$ for these secant slopes gives
$r_n<R=g_A-\kappa v_A$, hence $Y'(t)<R$.
Geometrically, the line $g=R+\kappa v$ passes through $(v_A,g_A)$
and lies above every point $(v_n,g_n)$ with $n\geq4$.
Since each $K_n$ increases with its argument, the same $R$ can then
be used in all curvature comparisons. We use the explicit quadratic
choice
\begin{equation}\label{eq:common-slope}
\kappa(Y,q)=3-6Y+20q-\frac83Y(1-Y),
\quad R(Y,q)=g_A-\kappa(Y,q)v_A,
\quad t=Y+q.
\end{equation}
When $q\geq2Y/5$, one has
$\kappa-71/24\geq(8Y-1)^2/24\geq0$.
For independent $t,Y$, direct simplification gives
\begin{equation}\label{eq:gA-product}
g_A(t,Y)=\frac{Y(Y+1)(Y+2)(Y+3)(5t+2Y+10)}
{6t(t+1)(t+2)(t+2Y+5)}.
\end{equation}
Let $G_A(r)=(\partial/\partial t+r\,\partial/\partial Y)g_A(t,Y)$.
Equivalently,
\begin{equation}\label{eq:gA-direction}
G_A(r)=g_A\left\{
r\sum_{j=0}^3\frac1{Y+j}+\frac{5+2r}{5t+2Y+10}
-\sum_{j=0}^2\frac1{t+j}-\frac{1+2r}{t+2Y+5}\right\}.
\end{equation}
The partial derivatives here are taken before substituting $t=Y+q$.
For the curvature inequalities it is useful to replace the quadratic
function $A(r)$ by an affine upper bound. Put
\begin{equation}\label{eq:B-F}
B(r)=G_A(r)-\frac3{20}(g_A-r),
\quad F_n(r)=H_{n-1}(t)-rH_n(Y).
\end{equation}
As proved below, $A(r)<B(r)$ on the required interval. The constant
$3/20$ preserves a definite part of the negative covariance while
removing the quadratic dependence on $r$; no sharp constant is needed.
We also use the rational functions
\begin{equation}\label{eq:eta-D}
\eta_n=-\frac{v_n}{g_n},
\quad D_n=(t+n)(\eta_{n+1}-\eta_n),\quad n\geq4,
\end{equation}
and put $U(Y)=3Y/7+3Y^2/14$.

\smallskip
\noindent\emph{The eight inequalities.}
The analytic argument requires the following eight rational
inequalities. The first four control the secant slopes; the last four
control curvature. The arguments below show why these initial
comparisons suffice for every index. Section~\ref{sec:algebra} proves
them on an independent parameter domain; positivity of $v_A$ is
required only on the curve.

\begin{proposition}[Finite rational inequalities]\label{lem:finite}
For independent real $Y,q$ satisfying
\begin{equation}\label{eq:closed-domain}
0<Y\leq1,\quad \frac{2Y}{5}\leq q\leq\frac Y2,
\quad Y-q\leq\frac12,\quad q\leq U(Y),
\end{equation}
the following eight inequalities hold:
\begin{align}
g_A-g_m-\kappa(v_A-v_m)&>0,\quad m=4,5,6,
\label{eq:slope-head}\\[7pt]
g_6-g_7-\kappa(v_6-v_7)&>0,
\label{eq:slope-edge}\\[7pt]
v_AF_4(R)-\eta_4B(R)&>0,
\label{eq:curvature-head}\\[7pt]
v_A(1-R)-B(R)D_m&>0,\quad m=4,5,6.
\label{eq:curvature-edge}
\end{align}
All functions are the rational expressions in
\eqref{eq:components}, \eqref{eq:block}, \eqref{eq:eta-D}, and
\eqref{eq:directional-components}--\eqref{eq:B-F}, with $t=Y+q$.
\end{proposition}

\smallskip
\noindent\emph{Applying the inequalities.}
The next two lemmas use these finite comparisons to bound the slope
and then make every term in the curvature estimate negative.

\begin{lemma}[Upper bound for the slope]\label{prop:slope-majorant}
Along the curve, for every $0<t\leq3/2$,
\begin{equation}\label{eq:slope-majorant}
0<Y'(t)<R(Y(t),q(t))<g_A<g_2<1.
\end{equation}
\end{lemma}

\begin{proof}
The domain is given by Corollary~\ref{cor:domain}.
Use $d_n,h_n$ from \eqref{eq:increments}. If $h_6\leq0$, then
$d_n-\kappa h_n>0$ for every $n\geq6$.
If $h_6>0$, \eqref{eq:slope-edge} gives $d_6/h_6>\kappa$;
Proposition~\ref{lem:increment-ratio} propagates this until the increments
become nonpositive, after which the preceding observation applies.
Thus $g_n-\kappa v_n$ decreases strictly for $n\geq6$.
The three signs \eqref{eq:slope-head} imply
$$
R>g_n-\kappa v_n,\quad n\geq4.
$$
Consequently
$$
R-r_n=\frac{v_A}{v_A-v_n}
\{g_A-g_n-\kappa(v_A-v_n)\}>0.
$$
Taking the average in \eqref{eq:trace-pairs} proves $Y'(t)<R$.
Since $\kappa,v_A>0$, $R<g_A<g_2$.
The last bound is \eqref{eq:g2-unit}.
\end{proof}

\begin{lemma}[Negativity of $K_n$]\label{lem:negative-kernels}
At every point of the curve with $0<t\leq3/2$, put
$R=R(Y(t),q(t))$. Then
$$
K_n(R)<0,\quad n\geq4.
$$
\end{lemma}

\begin{proof}
Corollary~\ref{cor:domain} and Proposition~\ref{lem:trace-sign} give
$v_A>0$ and $v_n<0$ for $n\geq4$.
Lemma~\ref{prop:slope-majorant} gives $0<R<g_A<g_2<1$.
First observe the exact identity
$$
A(r)=G_A(r)-\left\{\frac{D(r)}{1+\lambda}+\frac1{16}\right\}(g_A-r).
$$
For $0<r<1$, $0<t\leq3/2$, $0<Y\leq1$,
$$
\frac{D(r)}{1+\lambda}
=\frac{Y+3+r(1-t)}{(Y+t+2)(2Y+t+5)}
>\frac{Y+5/2}{(Y+7/2)(2Y+13/2)}\geq\frac{14}{153}.
$$
The last inequality is the identity
$$
153(Y+5/2)-14(Y+7/2)(2Y+13/2)=4(1-Y)(7Y+16)\geq0.
$$
Since $14/153+1/16-3/20=49/12240>0$, this proves the affine comparison
\begin{equation}\label{eq:A-B}
A(r)<B(r),\quad 0<r<\min(1,g_A).
\end{equation}

Set $L_n=\eta_nA(R)-v_AF_n(R)$.
By Proposition~\ref{lem:curvature-increments}, $\eta_4>0$
and $D_m>0$ for $m=4,5,6$. Hence the finite signs
\eqref{eq:curvature-head}--\eqref{eq:curvature-edge}, together with
\eqref{eq:A-B}, imply
$$
L_4<0,\quad v_A(1-R)-A(R)D_m>0,\quad m=4,5,6.
$$
Also
$$
F_{n+1}(R)-F_n(R)=\frac1{t+n}-\frac R{Y+n+1}
>\frac{1-R}{t+n}>0.
$$
If $A(R)\leq0$, the increase of $\eta_n$ and the strict increase
of $F_n(R)$ show that $L_n$ decreases strictly. If $A(R)>0$,
Proposition~\ref{lem:curvature-increments} gives
$D_n\leq\max(D_4,D_5,D_6)$ for every $n\geq4$, whence
$$
v_A\{F_{n+1}(R)-F_n(R)\}
-A(R)(\eta_{n+1}-\eta_n)
>\frac{v_A(1-R)-A(R)D_n}{t+n}>0.
$$
Thus $L_n<0$ for every $n\geq4$. Since $G_n(R)=-g_nF_n(R)$ and
$v_n=-g_n\eta_n$, it follows that
$K_n(R)=g_nL_n/(v_A-v_n)<0$.
\end{proof}

\subsection{Proof of the main theorem}\label{sec:main-proof}

\begin{proof}[Proof of Theorem~\ref{thm:main}]
\emph{Existence and regularity.}
For $t=\alpha+1$, the substitution $b=t-2Y$ maps $Y>0$
bijectively onto $b<\alpha+1$ and turns \eqref{eq:contact-bessel}
into $I_{\alpha,b}(j_{\alpha,2})=0$.
Lemma~\ref{lem:contact} gives existence, uniqueness, and analyticity;
$Y(3/2)=1$ gives $\beta(1/2)=-1/2$.

\emph{Strict curvature.}
Fix $0<t\leq3/2$ and write $R=R(Y(t),q(t))$.
Lemma~\ref{prop:slope-majorant} gives $0<Y'(t)<R<g_2$.
Lemma~\ref{lem:curvature-reduction} gives the reduction and strict
increase of $K_n$; Lemma~\ref{lem:negative-kernels} gives $K_n(R)<0$.
Consequently,
$$
Y''(t)\leq\sum_{n=4}^\infty\theta_nK_n(Y'(t))
<\sum_{n=4}^\infty\theta_nK_n(R)<0.
$$
The series converge absolutely and $\theta_n>0$, so the strict
inequalities are preserved on summation. All three lemmas include
$t=3/2$. By \eqref{eq:curvature-conversion},
$$
\beta''(\alpha)=-2Y''(\alpha+1)>0,
\quad -1<\alpha\leq\frac12.
$$

\emph{The left endpoint.}
Since $0<Y(t)<t$, one has $|\beta(t-1)|<t$, giving the continuous
extension $\beta(-1)=0$. Passing to the limit in the convexity
inequality proves convexity on $[-1,1/2]$. Equality at an interior
point of a chord would force affinity on that segment, contrary to
$\beta''(\alpha)>0$ in its interior. The extension is therefore
strictly convex.
\end{proof}

It remains to justify the estimates invoked above.
Section~\ref{sec:bounds} bounds the Bessel zero and the exponent.
Section~\ref{sec:weight-ratios} proves the weight comparisons used
in Lemma~\ref{lem:curvature-reduction}.
Section~\ref{sec:trace} establishes the signs, the parameter domain,
and the comparisons between consecutive terms needed in
Lemmas~\ref{prop:slope-majorant} and~\ref{lem:negative-kernels}.
Section~\ref{sec:algebra} proves Proposition~\ref{lem:finite}, with
the exact polynomial calculations given in Appendix~\ref{app:mathematica}.
None of these arguments assumes the curvature conclusion.

\section{Analytic bounds for the Bessel zero and the exponent}
\label{sec:bounds}

We need bounds for the moving Bessel zero and for the position and slope
of the curve. The zero estimates enter both the slope argument
below and the variation of the weights in
Proposition~\ref{lem:weight-ratios}. The bounds for the curve will restrict the
parameters in the rational inequalities.

\subsection{Bounds for the Bessel zero}

\begin{proposition}\label{lem:bessel-bounds}
For $0<t\leq3/2$, put $x(t)=j_{t-1,2}$,
$X(t)=x(t)^2/4$, and $c(t)=X'(t)/X(t)$.  Then
\begin{equation}\label{eq:bessel-bounds}
\frac{17}{5}<X(t)<4(t+1),\quad
\frac{188}{259+94t}<c(t)<\frac2{t+2}.
\end{equation}
The function $X$ extends analytically to $t=0$, with
$X(0)=j_{1,1}^2/4$ and $c(0)<1$.  Moreover,
\begin{equation}\label{eq:left-bessel-value}
J_0(j_{1,1})<-\frac25.
\end{equation}
\end{proposition}

\begin{proof}
We begin by continuing the second zero to $t=0$. The series
$$
F(t,z)=t+\sum_{k=1}^{\infty}
\frac{(-z)^k}{k!(t+1)_{k-1}}
$$
is jointly analytic near $t=0$ and on bounded $z$ sets.  For $t>0$ it
equals $t\,{}_0F_1(;t;-z)$, whereas
$F(0,z)=-\sqrt z\,J_1(2\sqrt z)$.  At $(t,z)=(0,0)$ its
$z$ derivative is $-1$, so the implicit-function theorem gives a
unique small zero $z_1(t)=t+O(t^2)>0$.  If
$X_0=j_{1,1}^2/4$, the function $F(0,z)$ is negative on $(0,X_0)$
and has a simple zero at $X_0$.  Uniform convergence on intervening
compact intervals, together with local uniqueness at the two endpoint
zeros, shows that the zero continuing $X_0$ is the second positive zero
for small $t>0$.  This proves the asserted analytic extension.

Let $z_k=j_{1,k}^2/4$. The product representation
\cite[(10.21.15)]{DLMF} for ${}_0F_1(;2;-z)$ gives
$$
\sum_{k\geq1}z_k^{-2}=\frac1{12},\quad
\sum_{k\geq1}z_k^{-3}=\frac1{48}.
$$
Indeed, its first three elementary symmetric coefficients are
$1/2$, $1/12$, and $1/144$, and the identities for the second and third
power sums give the displayed values.  Thus
\begin{equation}\label{eq:X0-bounds}
X_0>2\sqrt3>\frac{17}{5},\quad
X_0>48^{1/3}>\frac{29}{8}.
\end{equation}
The zeros increase strictly with the order \cite[(10.21.17)]{DLMF},
so $X(t)>X_0$.
To prove \eqref{eq:left-bessel-value}, put $x_*=2\sqrt{18/5}<j_{1,1}$.
Since $J_0'=-J_1<0$ before $j_{1,1}$ \cite[(10.6.3)]{DLMF},
the alternating Bessel series,
whose term magnitudes decrease after index one at $x_*$, gives
\begin{equation}\label{eq:J0-rational-bound}
J_0(j_{1,1})<J_0(x_*)<
\sum_{k=0}^{6}\frac{(-1)^k(18/5)^k}{(k!)^2}
=-\frac{157106}{390625}<-\frac25.
\end{equation}

The upper bound for $X$ follows from a two-dimensional trial space and
the variational principle for quadratic forms
\cite[Theorem~XIII.2]{ReedSimon1978}. Consider the form
$$
a[u]=\int_0^1 s^{2t-1}|u'(s)|^2\,ds
$$
in $L^2((0,1),s^{2t-1}ds)$, whose norm we denote by $\|\cdot\|$.
Its domain consists of locally absolutely
continuous functions with $u(1)=0$ and
$\int_0^1s^{2t-1}(|u|^2+|u'|^2)\,ds<\infty$; no essential boundary
condition is imposed at zero. Completeness on compact subintervals and
in the weighted norms shows that the form is closed. By
Cauchy--Schwarz and integration in $s$,
$$
|u(s)|^2\leq a[u]\int_s^1r^{1-2t}\,dr,
\quad \|u\|^2\leq\frac{a[u]}{4t}.
$$
The form domain is dense, since it contains every smooth function
compactly supported in $(0,1)$. The representation theorem for closed
positive forms \cite[Theorem~VIII.15]{ReedSimon1972} therefore gives a
unique positive self-adjoint operator; the preceding bound makes it
invertible. Integration by parts gives its natural boundary condition
$$
\lim_{s\downarrow0}s^{2t-1}u'(s)=0.
$$
Solving the inhomogeneous equation with this condition and $u(1)=0$
gives the symmetric inverse kernel
$$
G(s,v)=\int_{\max\{s,v\}}^1r^{1-2t}\,dr
$$
with respect to the weighted measure $v^{2t-1}dv$.
The integral of its square over the product measure is
$1/t$ times
$\int_0^1s^{4t-1}(\int_s^1r^{1-2t}dr)^2ds<\infty$.
Approximation of this kernel in the product $L^2$ space by finite
sums of separated functions makes its integral operator compact.
The spectrum is therefore discrete
\cite[Theorem~XIII.64]{ReedSimon1978}. In the eigenvalue equation
$$
-(s^{2t-1}u')'=\lambda s^{2t-1}u
$$
the zero-flux condition selects the solution proportional to
$s^{1-t}J_{t-1}(\sqrt\lambda s)$: the other independent solution has
nonzero limiting flux. This also excludes that solution when $0<t<1$,
although it then has finite form energy. The condition at one gives
the eigenvalues
$j_{t-1,m}^2$, $m\geq1$.

To construct the required trial space, take
$\phi_i(s)=s^{2i}(1-s^2)$,
$i=0,1,2$, and form the real symmetric matrices
$$
A_{ij}=\int_0^1s^{2t-1}\phi_i'\phi_j'\,ds,
\quad B_{ij}=\int_0^1s^{2t-1}\phi_i\phi_j\,ds.
$$
Direct integration gives
$$
\begin{aligned}
B_{ij}&=\frac1{2(t+i+j)}-\frac1{t+i+j+1}
+\frac1{2(t+i+j+2)},\\[7pt]
A_{ij}&=2\left\{\frac{ij}{t+i+j-1}
-\frac{i(j+1)+(i+1)j}{t+i+j}
+\frac{(i+1)(j+1)}{t+i+j+1}\right\},
\end{aligned}
$$
where a term with zero numerator is interpreted as zero.
The three leading principal minors of $A-16(t+1)B$ are
$$
\begin{aligned}
\Delta_1&=\frac{2(t^2-6t-8)}{t(t+1)(t+2)}<0,\\[7pt]
\Delta_2&=\frac{16(t^2-9t-4)}
{(t+1)(t+2)^2(t+3)^2(t+4)}<0,\\[7pt]
\Delta_3&=\frac{1152p(t)}
{t(t+1)(t+2)^3(t+3)^3(t+4)^3(t+5)^2(t+6)},
\end{aligned}
$$
where
$$
p(t)=t^5-7t^4-24t^3-44t^2+96t+64.
$$
On $[0,3/2]$ this polynomial is strictly concave, because
$p''(t)=20t^3-84t^2-144t-88<0$; also
$p(0)=64$ and $p(3/2)=5/32$.  Hence $\Delta_3>0$.
Successive completion of squares gives diagonal coefficients
$\Delta_1$, $\Delta_2/\Delta_1$, and $\Delta_3/\Delta_2$, with
signs $-,+,-$. Thus the matrix is negative definite on a
two-dimensional subspace.
The corresponding two-dimensional trial space satisfies
$a[u]<16(t+1)\|u\|^2$ for every nonzero $u$ in that space.
Its maximum Rayleigh quotient is strictly less than $16(t+1)$, so the
variational principle gives
$j_{t-1,2}^2<16(t+1)$.

To bound $c$, we use the endpoint estimates and the concavity of
the Bessel zero with respect to its order. This gives $x''(t)\leq0$; see
\cite{Elbert1977} and \cite[p.~155, Section~1]{Muldoon1986}.
Together with $x'(t)>0$ \cite[(10.21.17)]{DLMF}, this implies
\begin{equation}\label{eq:zero-reciprocal-slope}
c'(t)=\frac{2x''(t)}{x(t)}-\frac{c(t)^2}{2}
\leq-\frac{c(t)^2}{2},\quad
\left(\frac1{c(t)}\right)'\geq\frac12.
\end{equation}
It therefore suffices to bound $c$ at the two endpoints.

Put $x_0=j_{1,1}$. Implicit differentiation at order $-1$, using
the identities
\cite[(10.2.4), (10.15.3), and (10.5.2)]{DLMF}, gives
$$
x'(0)=\frac{1+J_0(x_0)^{-2}}{x_0},\quad
c(0)=\frac{1+J_0(x_0)^{-2}}{2X_0}<\frac{29}{8X_0}<1.
$$
The strict estimates follow from \eqref{eq:X0-bounds} and
\eqref{eq:J0-rational-bound}.

At the other endpoint, the half-order derivative identity
\cite[(10.15.6)]{DLMF} gives
$$
x'(3/2)=\operatorname{Si}(4\pi),\quad
c(3/2)=\frac{\operatorname{Si}(4\pi)}\pi,
\quad \operatorname{Si}(z)=\int_0^z\frac{\sin s}s\,ds.
$$
Here the integrand is assigned its limiting value at zero. The limit
$\operatorname{Si}(+\infty)=\pi/2$ is given in
\cite[(6.2.14)]{DLMF}.
Integration by parts gives
$$
\int_{4\pi}^\infty\frac{\sin s}s\,ds
=\frac1{4\pi}-2\int_{4\pi}^\infty\frac{\sin s}{s^3}\,ds
<\frac1{4\pi}.
$$
The last integral is positive by pairing each positive half-period with
the following negative half-period, whose denominator is larger.
Therefore, using $\pi>3$,
$$
c(3/2)>\frac12-\frac1{4\pi^2}>
\frac12-\frac1{36}>\frac{47}{100}.
$$
Integrating \eqref{eq:zero-reciprocal-slope} forward from zero and
backward from $3/2$, respectively, gives
$$
\frac1{c(t)}>1+\frac t2,\quad
\frac1{c(t)}<\frac{100}{47}-\frac{3/2-t}{2}
=\frac{259+94t}{188}.
$$
These are the two bounds in \eqref{eq:bessel-bounds}.
\end{proof}

\subsection{Bounds for the exponent and its derivative}

\begin{proposition}\label{lem:basic-bounds}
On $0<t\leq3/2$, the exponent satisfies
\begin{equation}\label{eq:basic-bounds}
\frac{2t}{3}\leq Y(t)<t,\quad 0<Y(t)\leq1,
\quad 0<q(t)=t-Y(t)\leq\frac12,
\quad \frac12<Y'(t)<1.
\end{equation}
The first inequality is strict for $t<3/2$.  In addition,
\begin{equation}\label{eq:affine-contact-bound}
Y(t)<\frac t2+\frac14,\quad 0<t<\frac32.
\end{equation}
\end{proposition}

\begin{proof}
The inequalities $0<Y<t$ and $0<Y\leq1$ follow from
Lemma~\ref{lem:contact}. The bound $Y\geq2t/3$ is precisely
\cite[Corollary~5.1(i)]{ChoChungYun2020} under
$\beta(t-1)=t-2Y(t)$; strictness for $t<3/2$ follows from
\cite[Theorem~5.1(ii)]{ChoChungYun2020}, since the Bessel integral
is strictly positive at $\beta= -t/3$. At $t=3/2$ one has $Y=1$.

The upper slope bound follows by averaging the component inequality
\begin{equation}\label{eq:unit-component-bound}
g_n<1+2v_n.
\end{equation}
To prove it for every $n\geq2$, set
$$
C=\frac{n+Y}{n},\quad D=\frac{2n+1}{n},\quad
H_n(Y)=\sum_{k=0}^n\frac1{Y+k}.
$$
For $0<Y\leq1$ one has $H_n(Y)>D/C$.  Indeed, the derivative of
$H_n(Y)-(2n+1)/(n+Y)$ is strictly negative, since
$$
-\sum_{k=0}^n\frac1{(Y+k)^2}+
\frac{2n+1}{(n+Y)^2}<-
\frac1{Y^2}+\frac{2n+1}{(n+Y)^2}<0.
$$
At $Y=1$ the difference equals
$\sum_{k=1}^{n+1}1/k-2+1/(n+1)$; its value at $n=2$ is $1/6$,
and its increment is $n/\{(n+1)(n+2)\}>0$.
Since $q>0$, expansion of a finite product gives
$$
e_n=\prod_{k=0}^n\left(1+\frac q{Y+k}\right)^{-1}
<\frac1{1+qH_n(Y)}<\frac C{C+qD}.
$$
The identity $g_n=e_n(t+n)/n$ therefore implies
$$
g_n+2qe_n=e_n(C+qD)<C,
$$
and hence
$$
1+2v_n-g_n>
1+\frac{2Y}{Y+n+1}-\frac{n+Y}{n}
=\frac{Y(n-Y-1)}{n(Y+n+1)}\geq0.
$$
This proves \eqref{eq:unit-component-bound}, including the strict sign
when $n=2$ and $Y=1$.  Averaging it in \eqref{eq:mean-trace} gives
$Y'<1$.  Thus $q'=1-Y'>0$, and $q(3/2)=1/2$ implies $q<1/2$
for $t<3/2$.

For the lower slope bound, the two defining equations remove the affine
part of a sequence of parameter derivatives. The resulting alternating
sum has a positive tail; its first pair will be estimated separately.
Here $t,Y,X$ are independent variables in the function
$$
M(t,Y,X)=\int_0^1u^{Y-1}{}_0F_1(;t;-Xu)\,du
=\sum_{k\geq0}(-1)^ka_k,
\quad a_k=\frac{X^k}{k!(t)_k(Y+k)}.
$$
All expressions are subsequently evaluated along the curve.
The vanishing integral and the vanishing endpoint value give
\begin{equation}\label{eq:two-alternating-moments}
\sum_{k\geq0}(-1)^ka_k=0,\quad
\sum_{k\geq0}k(-1)^ka_k=0.
\end{equation}
Also $\partial M/\partial Y=-A_0$, by integration by parts, with
$A_0>0$ as in \eqref{eq:weights}.  Put
$$
N=\frac{\partial M}{\partial t}
+X'\frac{\partial M}{\partial X}
+\frac12\frac{\partial M}{\partial Y},\quad
h_k=\sum_{j=0}^{k-1}\frac1{t+j}+\frac1{2(Y+k)}.
$$
Differentiating the defining equation and using
\eqref{eq:two-alternating-moments} gives
$$
N=A_0\left(Y'-\frac12\right)
=-\sum_{k\geq0}(-1)^ka_kh_k.
$$
The second forward difference of $h_k$ is $-k_n$, where
$$
k_n=\frac1{(t+n)(t+n+1)}
-\frac1{(Y+n)(Y+n+1)(Y+n+2)},\quad n\geq0.
$$
Subtracting the affine function $h_0+k(h_1-h_0)$ from $h_k$ and
then reversing the sums yields
\begin{equation}\label{eq:alternating-slope}
N=\sum_{n\geq0}(-1)^nC_nk_n,
\quad C_n=\sum_{\ell\geq0}(-1)^\ell(\ell+1)a_{n+2+\ell}.
\end{equation}
The coefficients $C_n$ are not the Taylor coefficients $c_n$
in \eqref{eq:positive-expansion}. Every rearrangement is absolutely
convergent, since $a_{k+1}/a_k=O(k^{-2})$ and the additional factors
have at most polynomial or logarithmic growth.  The two moment
identities give
\begin{equation}\label{eq:first-alternating-coefficients}
C_0=\frac1Y,\quad
C_1=\frac{X}{t(Y+1)}-\frac2Y.
\end{equation}

First, $k_0>k_1>\cdots>0$. Put $z=Y+n$. The bound $Y\geq2t/3$ gives
$0<q\leq Y/2\leq z/2$.  Replacing $q$ by its upper bound gives
$$
z(z+1)(z+2)-(z+q)(z+q+1)
\geq z\left(z^2+\frac34z+\frac12\right)>0,
$$
and
$$
\begin{aligned}
&2z(z+1)(z+2)(z+3)\\[7pt]
&\quad-3(z+q)(z+q+1)(z+q+2)\\[7pt]
&\quad\geq\frac z8(16z^3+15z^2+14z+24)>0.
\end{aligned}
$$
These are precisely the cross-multiplied inequalities $k_n>0$ and
$k_n-k_{n+1}>0$.

The alternating coefficients satisfy $C_n>0$ for $n\geq1$ and
$C_2>C_3>\cdots>0$.  Indeed,
$$
\frac{a_{k+1}}{a_k}
=\frac{X(Y+k)}{(k+1)(t+k)(Y+k+1)}
<\frac{4(t+1)}{(k+1)(t+k+1)},
$$
because $t-Y>0$ and $X<4(t+1)$. For $k\geq1$, the last bound increases
with $t$ and decreases with $k$. Using $t\leq3/2$ and $k=n+2+\ell$, for
$n\geq1$ the consecutive magnitudes in the series defining $C_n$
have ratio bounded above by
$$
\frac{20(\ell+2)}{(\ell+1)(\ell+4)(2\ell+11)}<1,
$$
since the difference between the denominator and the numerator is
$2\ell^3+21\ell^2+43\ell+4>0$.  Thus $C_n>0$ by the
alternating-series estimate.  Moreover,
$$
C_n-C_{n+1}=\sum_{\ell\geq0}(-1)^\ell(2\ell+1)a_{n+2+\ell}.
$$
For $n\geq2$ its consecutive magnitudes have ratio bounded above by
$$
\frac{20(2\ell+3)}{(2\ell+1)(\ell+5)(2\ell+13)}<1,
$$
the difference being $4\ell^3+48\ell^2+113\ell+5>0$.
This proves strict decrease. Pairing consecutive terms now gives
\begin{equation}\label{eq:positive-alternating-tail}
\sum_{n\geq2}(-1)^nC_nk_n>0.
\end{equation}

It remains to estimate the first pair. Put
$$
L=\frac{4Y(t+1)}{t(Y+1)}-2,\quad G=k_0-Lk_1.
$$
We show $G>0$ on the larger domain
$0<t\leq3/2$, $2t/3\leq Y\leq t$.
Clearing the positive denominators gives
\begin{equation}\label{eq:first-pair-polynomial}
G=\frac{p(t,Y)}{tY(t+1)(t+2)(Y+1)^2(Y+2)(Y+3)},
\end{equation}
where
$$
\begin{aligned}
p(t,Y)={}&-(t+2)Y^5-(3t+10)Y^4+(7t-10)Y^3\\[7pt]
&+(t^3+7t^2+41t+18)Y^2\\[7pt]
&+(-6t^3-18t^2+6t+12)Y-3t(t+1)(t+2).
\end{aligned}
$$
For fixed $t$, its third $Y$ derivative decreases on $Y\geq0$, because
$$
\frac{\partial^4p}{\partial Y^4}=-120(t+2)Y-72t-240<0.
$$
At the left endpoint,
$$
\frac{\partial^3p}{\partial Y^3}\left(t,\frac{2t}{3}\right)
=-\frac23(40t^3+152t^2+177t+90)<0.
$$
Therefore $\partial p/\partial Y$ is strictly concave on the required
interval, and its value at the left endpoint is positive:
$$
\begin{aligned}
\frac{\partial p}{\partial Y}\left(t,\frac{2t}{3}\right)
&=\frac2{81}(486+1215t+945t^2+33t^3-170t^4-40t^5)\\[7pt]
&\geq\frac2{81}\left(486+1215t+\frac{855}{2}t^2+33t^3\right)>0.
\end{aligned}
$$
Here $t^4\leq9t^2/4$ and $t^5\leq27t^2/8$.
Since this concave derivative can change sign only from positive to
negative, $p(t,Y)$ attains its minimum at an endpoint. There,
$$
\begin{aligned}
p(t,t)&=t(t+1)^2(t+2)(3-t^2)>0,\\[7pt]
p\left(t,\frac{2t}{3}\right)
&=\frac t{243}(486+729t+63t^2-192t^3-100t^4-32t^5)\\[7pt]
&\geq\frac t{243}(486+729t-558t^2)>0.
\end{aligned}
$$
For the second bound use $t^3\leq3t^2/2$, $t^4\leq9t^2/4$,
and $t^5\leq27t^2/8$.  The last quadratic is concave with endpoint
values $486$ and $324$ on $[0,3/2]$.  This proves $G>0$.

By \eqref{eq:first-alternating-coefficients} and $X<4(t+1)$,
$C_1/C_0<L$.  Since $C_0,k_1>0$, we conclude that
$$
C_0k_0-C_1k_1>C_0(k_0-Lk_1)>0.
$$
Together with \eqref{eq:alternating-slope} and
\eqref{eq:positive-alternating-tail}, this gives
$A_0(Y'-1/2)>0$, and therefore $Y'>1/2$.
The argument includes $t=3/2$ itself. Finally, integrate this strict
slope bound backward from $Y(3/2)=1$ to obtain
\eqref{eq:affine-contact-bound}.
\end{proof}
\section{Increasing ratios of consecutive weights}
\label{sec:weight-ratios}

The curvature argument requires the logarithmic derivatives of the
positive weights to increase strictly with the index. The next
proposition proves this property. We treat the indices $2$ and $3$
separately: the first ratio is explicit, and the second contains the
moving Bessel zero and requires the lower bound $\delta_3>1/16$.
A single integral comparison covers every index $n\geq4$. These
inequalities give the sign of the covariance term in
Lemma~\ref{lem:curvature-reduction}. The formula for $\delta_2$ and
the lower bound for $\delta_3$ give the two explicit contributions
used there. The proof is entirely analytic.

\begin{proposition}\label{lem:weight-ratios}
Let $w_n(t)$ be the positive weights in \eqref{eq:weights}.  Then
\begin{equation}\label{eq:weight-increments}
\delta_n(t):=\frac{d}{dt}\log\frac{w_{n+1}(t)}{w_n(t)}>0,
\quad n\geq2,\quad 0<t\leq\frac32.
\end{equation}
Moreover, $\delta_3(t)>1/16$ throughout this interval.
\end{proposition}

\begin{proof}
Fix a point of the branch and write $Y=Y(t)$, $X=X(t)$, $q=t-Y$,
$r=Y'(t)$, and $c=X'(t)/X(t)$.  We use the bounds of
Propositions~\ref{lem:basic-bounds} and \ref{lem:bessel-bounds}:
\begin{equation}\label{eq:mlr-inputs}
\begin{gathered}
0<Y\leq1,\quad 0<q\leq\frac12,\quad \frac12<r<1,
\\[7pt]
\frac{17}{5}<X<4(t+1),
\quad \frac{188}{259+94t}<c<\frac2{t+2}.
\end{gathered}
\end{equation}

\smallskip
\noindent\emph{The first ratio, $n=2$.}
By \eqref{eq:first-ratio},
$$
\delta_2(t)=\frac{r+1}{Y+t+2}-\frac r{Y+3}
=\frac{Y+3+r(1-t)}{(Y+t+2)(Y+3)}>0.
$$
For $t\leq1$ the numerator is positive term by term.  For
$1<t\leq3/2$, it is larger than $Y+4-t>0$, because $r<1$.

\smallskip
\noindent\emph{The second ratio, $n=3$.}
Here we need the stronger bound $\delta_3>1/16$, not just positivity.
The hypergeometric differential equation, or direct substitution of the
coefficient formula \eqref{eq:coeff-integral}, gives
$$
\frac{c_3}{c_2}=\frac{Y+t+2}{3},\quad
\frac{c_4}{c_2}=\frac{F}{12},\quad
F=Y^2+Yt+t^2+5Y+4t+6-X.
$$
In particular $F>0$.  Since $B(Y,5)/B(Y,4)=4/(Y+4)$,
\begin{equation}\label{eq:mlr-index3-ratio}
\frac{w_4}{w_3}=\frac{F}{D},\quad D=(Y+t+2)(Y+4).
\end{equation}
Differentiation gives $\delta_3=N/(FD)$, where
$$
\begin{aligned}
N={}&(Y+4)(X+2Yt+Y+t^2+4t+2)+rH-X'D,\\[7pt]
H={}&X(2Y+t+6)+Y^2-2Yt^2+4Y-t^3-6t^2-2t+4.
\end{aligned}
$$
Here $F,D>0$, so the problem is to obtain a uniform lower bound for
$N$.  The estimate $X>17/5$ first gives
$$
H>Y^2-2Yt^2+\frac{54}{5}Y-t^3-6t^2+
\frac75t+\frac{122}{5}>\frac{77}{8}.
$$
Indeed, the first expression on the right increases with $Y\geq0$;
at $Y=0$ it is concave in $t$, with endpoint values $122/5$ and
$77/8$ on $[0,3/2]$.
Using $r>1/2$ and $X'<2X/(t+2)$ then gives
$$
2(t+2)N>P(Y,t,X),
$$
where the following identity records the elimination explicitly:
$$
\begin{aligned}
P(Y,t,X)={}&X(t^2-4Y^2-16Y-4)
+(4t^2+11t+6)Y^2\\[7pt]
&+(24t^2+64t+32)Y-t^4+34t^2+80t+40.
\end{aligned}
$$
The coefficient of $X$ is negative on the parameter range.  Thus
$X<4(t+1)$ yields
$$
2(t+2)N>P_*(Y,t),
$$
where
\begin{equation}\label{eq:mlr-index3-pstar}
\begin{split}
P_*(Y,t)={}&(4t^2-5t-10)Y^2+(24t^2-32)Y\\[7pt]
&-t^4+4t^3+38t^2+64t+24.
\end{split}
\end{equation}
For the denominator, use the lower bound $X>17/5$ and put
$$
F_+(Y,t)=Y^2+Yt+t^2+5Y+4t+\frac{13}{5},
\quad E(Y,t)=5\{8P_*(Y,t)-(t+2)F_+(Y,t)D\}.
$$
We have $0<F<F_+$.  The inequality $E>0$ is equivalent to
$8P_*>(t+2)F_+D$; in particular $P_*>0$, and hence
$$
\delta_3=\frac{N}{FD}>
\frac{P_*}{2(t+2)F_+D}>\frac1{16}.
$$
The polynomial $P_*$ is strictly concave in $Y$ on $0\leq t\leq3/2$.
Every coefficient of $F_+D$ as a polynomial in $Y,t$ is nonnegative,
so $E$ is strictly concave in $Y$ as well.  It suffices to consider
the endpoints of $0\leq Y\leq\min\{t,1\}$: these are $Y=0,t$ when
$0\leq t\leq1$, and $Y=0,1$ when $1\leq t\leq3/2$.  Direct
factorization gives
$$
\begin{aligned}
E(0,t)&=-4(t+2)(15t^3-30t^2-207t-94),\\[7pt]
E(t,t)&=-2(t+1)(15t^4+75t^3+3t^2+278t-376),\\[7pt]
E(1,t)&=-5(13t^4+18t^3-330t^2-107t+402).
\end{aligned}
$$
The first expression is positive for $0\leq t\leq3/2$, since
$15t^3\leq(45/2)t^2$.  The polynomial in parentheses in the second
expression increases on $[0,1]$ and has value $-5$ at one.
For the last expression, the derivative of its parenthesized polynomial
is bounded above by
$$
52t^3+54t^2-660t-107\leq-462t-107<0,
\quad 1\leq t\leq\frac32,
$$
and that polynomial has value $-4$ at one.  Thus all endpoint values are
positive, and concavity proves $E>0$.  This establishes the uniform gap
$\delta_3>1/16$.

\smallskip
\noindent\emph{The remaining ratios, $n\geq4$.}
Fix such an index and put $a=Y+n$, $b=t+n$.  The proof rewrites the
weight ratio as a logarithmic derivative of one positive integral.  Its
variation in $(b,z)$ is controlled by a differential inequality, while its
variation in $a$ has a positive covariance.  For these comparisons
$a,b,z$ are independent variables; we substitute $z=X(t)$ only after
differentiation.  Define
$$
\phi_b(z)={}_0F_1(;b;-z),\quad
f_{a,b}(z)={}_1F_2(a;b,a+1;-z),
$$
$$
I(a,b,z)=\int_0^1v^{a-1}\phi_b(zv)\,dv=\frac{f_{a,b}(z)}a.
$$
The zero separation in Proposition~\ref{lem:zero-separation} shows that
$\phi_b(z)>0$ for $0\leq z\leq X$.  In particular $f_{a,b}$ is
positive there.  Put
$$
h(z)=-z\frac{f_{a,b}'(z)}{f_{a,b}(z)},\quad
S(z)=-z\frac{\phi_b'(z)}{\phi_b(z)},\quad
R_b(z)=\frac{\phi_{b+1}(z)}{\phi_b(z)}.
$$
Primes in these three expressions denote derivatives with respect to
$z$.  Differentiating under the integral gives
$f_{a,b}'(z)=-aI(a+1,b+1,z)/b$.  The coefficient formula
\eqref{eq:coeff-integral} and the beta identity
$B(Y,n+2)/B(Y,n+1)=(n+1)/(Y+n+1)$ now give
\begin{equation}\label{eq:mlr-euler-ratio}
\frac{w_{n+1}}{w_n}
=\frac{XI(a+1,b+1,X)}{b(a+1)I(a,b,X)}
=\frac{h(X)}{a+1}.
\end{equation}
It is therefore enough to show that the logarithmic growth of $h(X)$
dominates that of $a+1$.
Integration by parts gives
$$
(a-h)f_{a,b}=a\phi_b,
$$
and logarithmic differentiation yields the differential identity
\begin{equation}\label{eq:mlr-riccati}
z h'(z)=(a-h(z))(S(z)-h(z)).
\end{equation}

\smallskip
\noindent\emph{Variation in $b$ and $z$.}
We first control these two variables with $a$ fixed.
Comparison of the defining series gives
$$
R_b(z)=\frac1{1-zR_{b+1}(z)/\{b(b+1)\}}
=\sum_{k\geq0}A_k(b)z^k.
$$
Formal induction shows that $A_k(b)$ is a polynomial with nonnegative
coefficients, homogeneous of degree $k$, in the numbers
$1/\{(b+j)(b+j+1)\}$, $0\leq j<k$.  Consequently
\begin{equation}\label{eq:mlr-coefficient-derivative}
A_k'(b)\geq-k\left(\frac1b+\frac1{b+1}\right)A_k(b).
\end{equation}
These series and their $b$ derivatives converge locally uniformly for
$0\leq z\leq X$.  To see this, let $\lambda_1$ be the first positive
zero of $\phi_b$ and choose $X<R<\lambda_1$. The product
representation \cite[(10.21.15)]{DLMF}
shows that the closed disk $|z|\leq R$ contains no zero of $\phi_b$.
The same remains true for $b$ in a sufficiently small complex
neighborhood of its present value. Uniform coefficient estimates on this
disk \cite[Chapter~IV, Section~2, (2.14)]{Conway1978} justify the
asserted parameter differentiation.

Define the differential operator
$$
D_0=\frac{\partial}{\partial b}+cz\frac{\partial}{\partial z},
$$
holding $a$ fixed and treating the branch value $c$ as a fixed constant
in comparisons as $z$ varies.  Since
$S(z)=zR_b(z)/b$, coefficientwise application of
\eqref{eq:mlr-coefficient-derivative} gives
\begin{equation}\label{eq:mlr-S-derivative}
D_0S\geq\left(c-\frac1b-\frac1{b+1}\right)zS'(z)
+\frac{S(z)}{b+1}.
\end{equation}
Write $\lambda_k=j_{b-1,k}^2/4$. Logarithmic differentiation of
the same product and comparison of its coefficients give
$$
S(z)=\sum_{k\geq1}\frac{z}{\lambda_k-z},\quad
zS'(z)=\sum_{k\geq1}\frac{\lambda_kz}{(\lambda_k-z)^2},
\quad \lambda_1>b\sqrt{b+1}.
$$
Indeed, comparison of the quadratic terms in its logarithm gives
$\sum_k\lambda_k^{-2}=1/\{b^2(b+1)\}$, with more than one
positive summand. Comparison of corresponding summands gives
$$
\frac{S(z)}{zS'(z)}\geq1-\frac z{\lambda_1},\quad 0<z\leq X.
$$
Thus \eqref{eq:mlr-S-derivative} implies
\begin{equation}\label{eq:mlr-gamma}
D_0S\geq\gamma zS'(z),\quad
\gamma=c-\frac1b-\frac{X}{(b+1)\lambda_1}
>c-\frac1b-\frac{X}{b(b+1)^{3/2}},
\quad 0\leq z\leq X.
\end{equation}

The differential identity \eqref{eq:mlr-riccati} transfers this comparison
from $S$ to $h$. Put
$U=D_0h$, $V=zh'$, and $W=U-\gamma V$, with $\gamma$ fixed as in
\eqref{eq:mlr-gamma}.  The operators $D_0$ and $z\partial/\partial z$
commute.  Applying them to \eqref{eq:mlr-riccati} gives
$$
zW'-(2h-a-S)W=(a-h)(D_0S-\gamma zS')\geq0.
$$
Here $a-h>0$ by $(a-h)f_{a,b}=a\phi_b$.  An integrating factor is
$$
M(z)=z^a\exp\left\{\int_0^z\frac{S(s)-2h(s)}s\,ds\right\}.
$$
The integrand in the exponential is regular at zero.  Since $W(z)=O(z)$,
the regular solution satisfies
$$
W(z)=\frac1{M(z)}\int_0^z
\frac{M(s)}s(a-h(s))(D_0S(s)-\gamma sS'(s))\,ds\geq0.
$$
In particular $D_0h\geq\gamma zh'$.

To compare $zh'$ and $h$, note that
$$
a>1,\quad b=a+q\leq a+\frac12<a+1-\frac1{2a}.
$$
Under these conditions, Theorem~7.2(A) and the product representation
(1.3) in \cite{ChoChungPark2026}, applied to $f_{a,b}(s^2/4)$, give
$$
f_{a,b}(z)=\prod_{k\geq1}\left(1-\frac z{\zeta_k}\right).
$$
Here $0<\zeta_1\leq\zeta_2\leq\cdots$ are its zeros in the
$z$ variable, counted with multiplicity, and
$\sum_{k\geq1}\zeta_k^{-1}<\infty$.
Positivity of the integral representation on $[0,X]$ implies
$X<\zeta_1$.  Hence, for $0<z\leq X$,
\begin{equation}\label{eq:mlr-h-product}
h(z)=\sum_{k\geq1}\frac z{\zeta_k-z},\quad
zh'(z)=\sum_{k\geq1}\frac{\zeta_kz}{(\zeta_k-z)^2}>h(z)>0.
\end{equation}
The lower bound for $\gamma$ is uniform in the index. Since $b\geq t+4$,
$$
\sqrt{b+1}\geq\sqrt{t+5}>\frac{t+11}{5},
\quad 25(t+5)-(t+11)^2=4+3t-t^2>0.
$$
Equations \eqref{eq:mlr-inputs} and \eqref{eq:mlr-gamma} therefore give
$$
\begin{aligned}
\gamma-\frac1{t+9/2}
&>\frac{188}{259+94t}-\frac1{t+4}
-\frac{20(t+1)}{(t+4)(t+5)(t+11)}-\frac1{t+9/2}\\[7pt]
&=\frac{83455-5150t-19683t^2-3198t^3}
{(t+4)(t+5)(t+11)(2t+9)(94t+259)}>0.
\end{aligned}
$$
The numerator decreases on $[0,3/2]$ and equals $20650$ at
$t=3/2$; every denominator factor is positive. In particular $\gamma>0$,
so the comparison $D_0h\geq\gamma zh'$ and \eqref{eq:mlr-h-product}
imply
\begin{equation}\label{eq:mlr-logh-lower}
D_0\log h=\frac{D_0h}{h}>\gamma.
\end{equation}

\smallskip
\noindent\emph{Variation in $a$ and conclusion.}
The remaining variation, in $a$, has a positive sign. At $z=X$ define
the probability measure
$$
d\mu(v)=\frac{v^{a-1}\phi_b(Xv)}{I(a,b,X)}\,dv,
\quad 0<v<1.
$$
For integrable functions $g,k$, write
$$
E_\mu[g]=\int_0^1g(v)\,d\mu(v),\quad
\operatorname{Cov}_\mu(g,k)=E_\mu[gk]-E_\mu[g]E_\mu[k],
$$
whenever the displayed integrals exist.  The identity
$$
\frac{I(a+1,b+1,X)}{I(a,b,X)}=E_\mu[vR_b(Xv)]
$$
and differentiation in $a$ give
$$
\frac{\partial}{\partial a}\log h(X)
=\frac{\operatorname{Cov}_\mu(vR_b(Xv),\log v)}
{E_\mu[vR_b(Xv)]}>0.
$$
The strict sign follows by expanding the definition of covariance:
$$
\operatorname{Cov}_\mu(g,k)=\frac12
\int_0^1\int_0^1(g(v)-g(u))(k(v)-k(u))\,d\mu(u)\,d\mu(v):
$$
both $vR_b(Xv)$ and $\log v$ are strictly increasing.  The former
assertion follows from the nonnegative coefficients of $R_b$; the
density of $\mu$ is strictly positive.  All logarithmic integrals
converge because $a>0$.  Differentiation under the integral and
Fubini's theorem are justified, for example, by
\cite[Theorems~2.27 and~2.37(b)]{Folland1999}.

Finally, $a+1=b-q+1\geq t+9/2$. Along the branch,
$d/dt=r\partial/\partial a+D_0$, so
\eqref{eq:mlr-euler-ratio} and \eqref{eq:mlr-logh-lower} yield
$$
\begin{aligned}
\delta_n
&=r\frac{\partial}{\partial a}\log h(X)
+D_0\log h(X)-\frac r{a+1}\\[7pt]
&>\gamma-\frac1{t+9/2}>0.
\end{aligned}
$$
This proves the strict increment for every tail index. All three
arguments apply at $t=3/2$ itself, on the analytic branch through
that endpoint.
\end{proof}

\section{Parameter bounds and comparison of consecutive terms}\label{sec:trace}

This section proves the signs and inequalities between consecutive
terms used in Section~\ref{sec:curvature}. We first work with independent
parameters, then use the resulting inequalities to bound the curve.
Finally, two monotonicity properties extend the initial comparisons
in Proposition~\ref{lem:finite} to every index.
Throughout, $t=Y+q$ when $Y,q$ are independent variables, and
$g_A,v_A$ are defined by \eqref{eq:block}.

\subsection{Signs of the terms}

\begin{proposition}\label{lem:trace-sign}
Suppose $0<Y\leq1$ and $2/5\leq q/Y\leq1/2$. Then $v_n<0$
for every $n\geq4$. At a point of the curve satisfying these
bounds, $v_A>0$, the positive-average representation
\eqref{eq:trace-pairs} holds, and $0<Y'(t)<g_A<g_2$.
\end{proposition}

\begin{proof}
Set $a=q/Y$ and $b_n=u_n/e_n$. Direct cancellation gives
$$
\frac{b_{n+1}}{b_n}=\frac{t+n+1}{Y+n+2}<1,
\quad \frac{b_4}{q}=\frac{1+a}{a}
\frac{\prod_{j=1}^4((1+a)Y+j)}{\prod_{j=1}^5(Y+j)}=:T(Y,a).
$$
On the stated rectangle,
$$
\frac{\partial\log T}{\partial a}
=-\frac1{a(1+a)}+\sum_{j=1}^4\frac{Y}{(1+a)Y+j}
\leq-\frac43+\sum_{j=1}^4\frac1{j+1}=-\frac1{20}.
$$
At $a=2/5$,
$$
\frac{\partial\log T}{\partial Y}
=\frac25\sum_{j=1}^4\frac{j}{(7Y/5+j)(Y+j)}-\frac1{Y+5}>0.
$$
Indeed, the first three summands inside the sum exceed $1/6$,
whereas $1/(Y+5)\leq1/5$. Hence
$b_4/q\leq T(1,2/5)=11781/12500<1$, proving $v_n=e_n(b_n-q)<0$.
On the curve the second identity in \eqref{eq:mean-trace} reads
$$
Wv_A+\sum_{n=4}^\infty w_nv_n=0,
\quad W=w_2+w_3=(1+\lambda)w_2.
$$
It implies $v_A>0$; substitution into the definitions of $\theta_n$
and $r_n$ gives the two sums in \eqref{eq:trace-pairs}. Finally,
$$
\frac{g_{n+1}}{g_n}=
\frac{n(Y+n+1)}{(n+1)(t+n)}<1.
$$
Thus $g_n<r_n<g_A<g_2$ for $n\geq4$.
\end{proof}

\subsection{Comparing consecutive terms for the slope bound}

\begin{proposition}\label{lem:increment-ratio}
Under the hypotheses of Proposition~\ref{lem:trace-sign}, set
$$
d_n=g_n-g_{n+1},\quad h_n=v_n-v_{n+1},\quad n\geq4.
$$
Once $h_n$ becomes nonpositive it remains nonpositive. If $n\geq6$
and $h_n,h_{n+1}>0$, then
$$
\frac{d_{n+1}}{h_{n+1}}>\frac{d_n}{h_n}.
$$
\end{proposition}

\begin{proof}
The definitions give
\begin{equation}\label{eq:increments}
d_n=\frac{e_n(t+nq)}{n(n+1)},
\quad h_n=\frac{Y}{(Y+n+1)(Y+n+2)}-\frac{q^2e_n}{t+n+1}.
\end{equation}
Put
$$
z_n=\frac{q^2e_n(Y+n+1)(Y+n+2)}{Y(t+n+1)},
\quad \rho_n=\frac{Y+n+3}{t+n+2}>1.
$$
Then $z_{n+1}=\rho_nz_n$ and
$h_n=Y(1-z_n)/((Y+n+1)(Y+n+2))$, proving persistence of the sign.
When both increments are positive, cancellation gives
$$
\frac{d_{n+1}/h_{n+1}}{d_n/h_n}
=p_n\frac{1-z_n}{1-\rho_nz_n},
\quad p_n=\frac{n(Y+n+3)\{t+(n+1)q\}}
{(n+2)(t+n+1)(t+nq)}.
$$
It is enough to prove $p_n-1+z_n(\rho_n-p_n)>0$. Here
\begin{equation}\label{eq:increment-positive-factor}
\rho_n-p_n=
\frac{t(Y+n+3)(2Y+nq+n+2q+2)}
{(n+2)(t+nq)(t+n+1)(t+n+2)}>0.
\end{equation}
Writing $a=q/Y$, set
$$
C_n=a^2(n+1)(n+2)+2a(n+2)+2,
\quad E_n=a(n+2)(n-1)-2.
$$
Then
$$
p_n-1=-\frac{Y(YC_n-E_n)}{(n+2)(t+nq)(t+n+1)}.
$$
For $n\geq9$, the polynomial $C_n-E_n$ is convex in $a$.
Its endpoint values, multiplied by $25$ and $4$, are respectively
$-6n^2+22n+168$ and $-n^2+5n+30$, both negative.
Since $C_n>0$ and $Y\leq1$, also $YC_n-E_n<0$, so $p_n>1$.

For $6\leq n\leq8$, only the case $YC_n>E_n$ remains.
The inequality $v_n<0$ gives
$z_n>q(Y+n+2)/(t+n+1)=:z_n^0$.
After positive factors are cancelled,
$p_n-1+z_n^0(\rho_n-p_n)>0$ amounts to $\Psi_n(Y,a)<1$, where
$$
\Psi_n(Y,a)=\frac{(YC_n-E_n)(t+n+1)(t+n+2)}
{a(1+a)Y(Y+n+2)(Y+n+3)\{n+2+Y[2+a(n+2)]\}}.
$$
On the set $YC_n>E_n$, this function increases with $Y$.
For completeness, its logarithmic derivative is
$$
\begin{aligned}
\frac{\partial\log\Psi_n}{\partial Y}
={}&\frac{E_n}{Y(YC_n-E_n)}
+\frac{1+a}{t+n+1}+\frac{1+a}{t+n+2}\\[7pt]
&-\frac1{Y+n+2}-\frac1{Y+n+3}
-\frac{2+a(n+2)}{n+2+Y[2+a(n+2)]}.
\end{aligned}
$$
The first term is at least $13/16$. Indeed,
$$
29E_n-13C_n=(n+2)a\{(29-13a)n-55-13a\}-84\geq0;
$$
the expression increases with $n\geq4$, and at $n=4$ equals
$-6(5a-2)(13a-7)\geq0$. Moreover,
$0<Y(YC_n-E_n)\leq C_n-E_n\leq16E_n/13$.
The remaining terms give
$$
\frac{\partial\log\Psi_n}{\partial Y}
\geq\frac{13}{16}
+\frac75\left(\frac1{n+5/2}+\frac1{n+7/2}\right)
-\frac12-\frac3{n+2}-\frac1{n+3}>0.
$$
The last rational expression has value $7829/109200$ at $n=4$,
and forward difference
$$
\frac{6(4n^3+56n^2+234n+301)}
{5(n+2)(n+3)(n+4)(2n+5)(2n+9)}>0.
$$
At $Y=1$, the numerator of
$p_n-1+z_n^0(\rho_n-p_n)$, after positive factors are removed,
is respectively
$$
\begin{array}{c|l}
n&\text{numerator factor}\\[7pt]\hline
6&8(1-a)(7a^3+33a^2+284a-36)\\[7pt]
7&18(1-a)(4a^3+23a^2+223a-20)\\[7pt]
8&10(1-a)(9a^3+61a^2+656a-44).
\end{array}
$$
Each cubic increases with $a>0$ and is positive at $a=2/5$.
This proves the assertion.
\end{proof}

\subsection{The parameter domain along the curve}

The preceding comparisons will now restrict the position of the curve.
The resulting domain is the one required in
Proposition~\ref{lem:finite}.

\begin{proposition}\label{lem:ratio-bound}
Along the curve,
\begin{equation}\label{eq:ratio-bound}
\frac{Y(t)}t<\frac57,\quad 0<t\leq\frac32.
\end{equation}
\end{proposition}

\begin{proof}
By Proposition~\ref{lem:bessel-bounds}, $X(t)\to X_0=j_{1,1}^2/4$
as $t\downarrow0$.

Put $s(t)=Y(t)/t$. The defining equation gives
$$
0=\frac1{s(t)}+
\sum_{k=1}^\infty\frac{t(-X(t))^k}{k!(t)_k(ts(t)+k)}.
$$
Since $2/3\leq s(t)<1$, the summands are bounded in absolute value
by $L^k/(k!)^2$ on a short initial interval, for a fixed $L$.
Dominated convergence \cite[Theorem~2.24]{Folland1999} shows that
every convergent subsequence of
$s(t)$ has the same limit. Compactness of $[2/3,1]$ therefore gives
\begin{equation}\label{eq:ratio-left}
\lim_{t\downarrow0}s(t)=\frac1{1-J_0(j_{1,1})}<\frac57.
\end{equation}
The last inequality follows from the elementary bound
$J_0(j_{1,1})<-2/5$ proved in Proposition~\ref{lem:bessel-bounds}.

Suppose that $s=5/7$ at a point of the curve. At that point
\begin{equation}\label{eq:barrier-domain}
t=7Y/5,\quad q=2Y/5,\quad 0<Y<5/6.
\end{equation}
The last inequality follows by integrating $Y'>1/2$ backwards
from $Y(3/2)=1$. Proposition~\ref{lem:trace-sign} applies, so it suffices
to prove $r_n<5/7$ for every $n\geq4$.
Set
$$
Z_n=(v_A-v_n)d_n-(g_A-g_n)h_n.
$$
Subtraction gives
$$
r_{n+1}-r_n=-\frac{v_AZ_n}{(v_A-v_n)(v_A-v_{n+1})}.
$$
If $h_n\leq0$ the sign is negative. Otherwise put
$s_n=(g_A-g_n)/(v_A-v_n)$. Whenever $Z_n>0$, one has
$d_n/h_n>s_n$, and
$$
s_{n+1}=\frac{(v_A-v_n)s_n+d_n}{v_A-v_n+h_n}
$$
lies between $s_n$ and $d_n/h_n$. Proposition~\ref{lem:increment-ratio}
therefore propagates $Z_6>0$ to strict decrease of $r_n$ for $n\geq6$.

Here the starting sign has an elementary polynomial proof.
On \eqref{eq:barrier-domain}, exact expansion gives
$$
Z_6=\frac{625Y^2(Y+1)(Y+2)(Y+3)(2Y-5)F(Y)}
{1029(Y+7)(Y+8)(17Y+25)\prod_{k=1}^6(7Y+5k)^2},
$$
where
$$
\begin{aligned}
F(Y)={}&160768750Y^{11}+7196784375Y^{10}
+122777933929Y^9\\[7pt]
&+1111731243910Y^8+6006573685000Y^7
+20081074013875Y^6\\[7pt]
&+40431416015625Y^5+41717794150000Y^4
+1814329125000Y^3\\[7pt]
&-43824112500000Y^2-41465418750000Y-11930625000000.
\end{aligned}
$$
Every nonconstant term of $F(Y)/Y^3$ has positive derivative for
$Y>0$. Moreover,
$$
F(5/6)=-\frac{2795043934376107421875}{90699264}<0.
$$
Thus $F(Y)<0$ on \eqref{eq:barrier-domain}, and $Z_6>0$.
Consequently $\sup_{n\geq4}r_n=\max(r_4,r_5,r_6)$.

To check these three values, put
$$
D(Y)=(7Y+5)^2(7Y+10)^2(7Y+15)^2(17Y+25)
$$
and $T_n=-v_n(5/7-g_A)+v_A(5/7-g_n)$.
The identities
$$
\begin{aligned}
T_4&=\frac{5Y^2(5-2Y)F_4(Y)}
{588(Y+4)(Y+5)(7Y+20)D(Y)},\\[7pt]
T_5&=\frac{5Y^2(5-2Y)F_5(Y)}
{294(Y+4)(Y+6)(7Y+20)(7Y+25)D(Y)},\\[7pt]
T_6&=\frac{5Y^2(5-2Y)F_6(Y)}
{147(Y+4)(Y+7)(7Y+20)(7Y+25)(7Y+30)D(Y)}
\end{aligned}
$$
follow by expansion of the finite products in \eqref{eq:components}.
Their numerators are
$$
\begin{aligned}
F_4(Y)={}&3609375Y^9+51186625Y^8+356632500Y^7
+1624200018Y^6\\[7pt]
&+5136726915Y^5+11051480025Y^4+15434489250Y^3\\[7pt]
&+13056552500Y^2+5892525000Y+1030500000,
\end{aligned}
$$
$$
\begin{aligned}
F_5(Y)={}&12140625Y^{10}+199921750Y^9+1720639375Y^8
+10443869856Y^7\\[7pt]
&+45966978675Y^6+140214297150Y^5+284869787625Y^4\\[7pt]
&+371482047500Y^3+293296962500Y^2
+124359000000Y+20700000000,
\end{aligned}
$$
$$
\begin{aligned}
F_6(Y)={}&35546875Y^{11}+661124625Y^{10}+6729102625Y^9
+53094414709Y^8\\[7pt]
&+321318004495Y^7+1365848355175Y^6+3907924666875Y^5\\[7pt]
&+7386400099375Y^4+8991672431250Y^3+6679280062500Y^2\\[7pt]
&+2689588125000Y+429975000000.
\end{aligned}
$$
All their coefficients are positive, and all remaining factors have
the displayed sign. Hence $5/7-r_n=T_n/(v_A-v_n)>0$ for
$n=4,5,6$, and therefore for all $n\geq4$. Equation
\eqref{eq:trace-pairs} gives $Y'(t)<5/7$ whenever $s(t)=5/7$.
By \eqref{eq:ratio-left}, the first point where $s(t)=5/7$ would have
$s'(t)\geq0$, whereas $s'(t)=(Y'(t)-5/7)/t<0$ there.
The endpoint $t=3/2$ is excluded since $s(3/2)=2/3$.
This proves \eqref{eq:ratio-bound}.
\end{proof}

\begin{corollary}\label{cor:domain}
At every point of the curve with $0<t\leq3/2$,
\begin{equation}\label{eq:domain}
0<Y\leq1,\quad \frac{2Y}{5}<q\leq\frac Y2,
\quad Y-q\leq\frac12,\quad q<U(Y)=\frac{3Y}{7}+\frac{3Y^2}{14}.
\end{equation}
Moreover $v_A>0$.
\end{corollary}

\begin{proof}
Only the last bound needs proof. Put $a=q/Y$. Cancellation in
\eqref{eq:block} gives
$$
v_A=\frac{Y(Ya-1)V(Y,a)}
{(Y+4)(a+1)(t+1)(t+2)(t+3)(Ya+3Y+5)},
$$
where
$$
\begin{aligned}
V(Y,a)={}&-2Y^4a-3Y^4
+Y^3(a^4+7a^3+18a^2+2a-24)\\[7pt]
&+Y^2(13a^3+67a^2+56a-69)\\[7pt]
&+Y(60a^2+130a-84)+84a-36.
\end{aligned}
$$
For $a\geq0$, $0<Y\leq1$, its derivative in $a$ exceeds
$84-2Y^4>0$, and
$$
V\left(Y,\frac37+\frac{3Y}{14}\right)
=\frac{3Y(Y+2)^2}{38416}
\{27Y^4+990Y^3+8606Y^2+20720Y+2352\}>0.
$$
Since $Ya=q\leq1/2$ and $v_A>0$, one has $V(Y,a)<0$.
Strict increase in $a$ proves $a<3/7+3Y/14$.
\end{proof}

\subsection{Comparing consecutive terms for the curvature bound}

The final comparison extends the curvature inequalities in
Proposition~\ref{lem:finite} to every index, as used in
Lemma~\ref{lem:negative-kernels}.

\begin{proposition}\label{lem:curvature-increments}
On the parameter region $0<Y\leq1$, $2/5\leq q/Y\leq1/2$, the
quantities in \eqref{eq:eta-D} satisfy $\eta_n>0$ and $D_n>0$.
Moreover, $\eta_n$ is strictly increasing, and
$D_{n+1}<D_n$ for every $n\geq6$.
\end{proposition}

\begin{proof}
Again put $b_n=u_n/e_n$. Then
$\eta_n=n(q-b_n)/(t+n)$ is positive and increasing, since $b_n$
decreases. Direct subtraction gives
\begin{equation}\label{eq:D-difference}
\begin{aligned}
D_n&=\frac{qt}{t+n+1}+b_n\frac{(1-q)n-t}{Y+n+2},\\[7pt]
D_n-D_{n+1}
&=\frac{qt}{(t+n+1)(t+n+2)}
+\frac{b_nN_n}{(Y+n+2)(Y+n+3)},\\[7pt]
N_n&=n(1-q)^2+2q^2-2q-1+Y(2q-3).
\end{aligned}
\end{equation}
For $n\geq6$, we first show $b_n<4q/5$. It suffices to estimate
$$
\frac{b_6}{q}=\frac{1+a}{a}
\frac{\prod_{j=1}^6((1+a)Y+j)}{\prod_{j=1}^7(Y+j)},
\quad a=q/Y.
$$
This expression increases with $Y$: its logarithmic derivative is
$$
a\sum_{j=1}^6\frac{j}{((1+a)Y+j)(Y+j)}-\frac1{Y+7}>0,
$$
because each of its first three summands exceeds $1/15$.
Its logarithm is convex in $a$, since the second derivative is at least
$$
\frac1{a^2}-\frac1{(1+a)^2}
-\sum_{j=1}^6\frac{Y^2}{((1+a)Y+j)^2}
\geq\frac{32}{9}-\frac32>0.
$$
The two possible maximum values are consequently
$62271/78125$ and $6435/8192$, at $(Y,a)=(1,2/5)$ and $(1,1/2)$.
Both are smaller than $4/5$.

If $N_n\geq0$, \eqref{eq:D-difference} is positive. If $N_n<0$,
the preceding bound reduces its positivity to that of
$$
V_n(Y,q)=5(Y+q)(Y+n+2)(Y+n+3)
+4N_n(Y+q+n+1)(Y+q+n+2).
$$
This polynomial is concave in $Y$ on $0\leq Y\leq1$,
$0\leq q\leq1/2$, since
$$
\frac{\partial^3V_n}{\partial Y^3}=48q-42<0,
\quad \frac{\partial^2V_n}{\partial Y^2}(0,q)
\leq-10n-21<0.
$$
At $Y=0,1$, the respective coefficients of $q^2,q^3,q^4$ are
$$
\begin{array}{c|ccc}
Y&[q^2]V_n&[q^3]V_n&[q^4]V_n\\[7pt]\hline
0&4n^3+4n^2-4n-12&8n^2+20n+16&4n+8\\[7pt]
1&4n^3+12n^2+28n+32&8n^2+28n+40&4n+8.
\end{array}
$$
They are positive for $n\geq6$, so the endpoint polynomials are
convex in $q$. Their derivatives at $q=1/2$ are
$-4n^3-9n^2+2n+6$ and $-4n^3-9n^2+26n+46$, both negative.
Thus they decrease on $[0,1/2]$. Writing $n=6+z$, their minimum
values are
$$
V_n(0,1/2)=z^3+\frac{37}{2}z^2+\frac{425}{4}z+180,
\quad V_n(1,1/2)=z^3+\frac{35}{2}z^2+\frac{317}{4}z+29.
$$
These are positive for $z\geq0$. Concavity in $Y$ completes the proof.
\end{proof}
\section{The finite algebraic inequalities}\label{sec:algebra}

We prove the eight inequalities of Proposition~\ref{lem:finite}.
Monotonicity and concavity reduce them to finitely many polynomial
sign checks, as shown below. All polynomials are defined by rational
expressions in independent variables $Y,q$, with $t=Y+q$.
Appendix~\ref{app:mathematica} reproduces their construction and the
exact coefficient calculations. No approximations to the curve or
to special functions enter this argument.

For a polynomial $p(z)=\sum_{i=0}^d a_i z^i$, write
\begin{equation}\label{eq:bernstein-conversion}
p(z)=\sum_{j=0}^d b_j\binom dj z^j(1-z)^{d-j},
\quad b_j=\sum_{i=0}^j a_i\frac{\binom ji}{\binom di}.
\end{equation}
This follows by expanding $(z+(1-z))^{d-i}$ in each monomial.
Since the basis functions are nonnegative and sum to one on $[0,1]$,
positive $b_j$ imply $p>0$. For coefficients that are themselves
polynomials in $s$, we also use
\begin{equation}\label{eq:row-lower}
\sum_{j=0}^e c_js^j\geq
c_0+\sum_{j=1}^e\min\{0,c_j\},
\quad 0\leq s\leq1.
\end{equation}
Thus every sign check reduces to rational arithmetic.

The vertical sections of \eqref{eq:closed-domain} are
$\ell(Y)\leq q\leq h(Y)$, where
\begin{equation}\label{eq:vertical-sections}
\ell(Y)=\max\left(\frac{2Y}{5},Y-\frac12\right),
\quad h(Y)=\min\left(\frac Y2,U(Y)\right),\quad 0<Y\leq1.
\end{equation}
The lower boundary changes at $Y=5/6$, and the upper at $Y=1/3$.

\subsection{Slope}

Denote the left sides of \eqref{eq:slope-head} and
\eqref{eq:slope-edge} by $T_4,T_5,T_6,E$, respectively. Set
$C=3Y+q+5$ and define
\begin{equation}\label{eq:slope-polynomials}
\begin{aligned}
S_m&=L_mT_m/Y,\quad
L_m=a_m(t)_{m+1}(Y+4)(Y+m+1)C,\quad m=4,5,6,\\[7pt]
S_E&=L_EE/Y,\quad L_E=42(t)_8(Y+7)(Y+8),
\quad (a_4,a_5,a_6)=(12,30,6).
\end{aligned}
\end{equation}
These are integer polynomials and all $L_m,L_E$ are positive.

For fixed $Y$, the appendix verifies
$\partial^3S_m/\partial q^3>0$ on $0\leq q\leq Y/2$.
At $q=h(Y)$ it verifies
$\partial S_m/\partial q<0$ when $Y\geq1/3$, and also when
$m=6$, $1/4\leq Y\leq1/3$; in the remaining cases it verifies
$\partial^2S_m/\partial q^2<0$.
Either upper-endpoint condition puts the minimum of $S_m$ at an
endpoint of \eqref{eq:vertical-sections}. Indeed, the second makes
$S_m$ concave. Under the first, its first derivative is convex;
after any zero it is negative up to the upper endpoint, so $S_m$
has no interior minimum.

The same calculation gives $S_m(Y,\ell(Y))>0$ and
$S_m(Y,h(Y))>0$. These endpoint and derivative checks use
\eqref{eq:bernstein-conversion}, with an affine change of variable
on each indicated interval. The factor $Y$ at the excluded endpoint
$Y=0$ is removed first. The exact intervals, including the subdivision
at $5/6$ for $S_5(Y,Y/2)$ and $S_6(Y,Y/2)$, are specified in the code.
The third derivatives are bounded below by discarding their positive
terms containing $q$ and replacing $q^j$ by $(Y/2)^j$ in their
negative terms; the resulting polynomials have positive coefficients.
Consequently $S_4,S_5,S_6>0$ on the domain.

For $S_E$, substitute $q=Y(2/5+s/10)$ and expand
$S_E(Y,Y(2/5+s/10))/Y$ in the Bernstein basis in $Y$.
Applying \eqref{eq:row-lower} to its twelve coefficient polynomials
gives positive lower bounds. Hence $S_E>0$ as well.

\subsection{Curvature}

Denote the left sides of \eqref{eq:curvature-head} and
\eqref{eq:curvature-edge} by $E_0,E_4,E_5,E_6$, respectively. Define
\begin{equation}\label{eq:curvature-polynomials}
P_i=\frac{\Delta_iE_i}{Y(1-q)},\quad i=0,4,5,6,
\end{equation}
where the positive denominators are
\begin{equation}\label{eq:curvature-denominators}
\begin{aligned}
\Delta_0={}&90(Y+1)(Y+2)(Y+3)(Y+5)(Y+4)^3(t+4)C^3(t)_4^2,\\[7pt]
\Delta_m={}&180(Y+1)(Y+2)(Y+3)(Y+4)^2(t+m+1)C^3(t)_4^2
(Y+5)_{m-2},\\[7pt]
&\hfill m=4,5,6.
\end{aligned}
\end{equation}
The appendix checks these polynomial identities directly.
Since $1-q\geq1/2$, it suffices to prove $P_i>0$.

Apply the same substitution $q=Y(2/5+s/10)$ to
$-\partial^2P_i/\partial q^2$. Bernstein expansion in $Y$,
followed by \eqref{eq:row-lower}, gives strictly positive lower
bounds for every coefficient polynomial in $s$. Thus each $P_i$
is strictly concave in $q$, and its positivity reduces to the four
boundary arcs of \eqref{eq:vertical-sections}. They are parameterized by
\begin{equation}\label{eq:four-arcs}
\begin{aligned}
&\left(\frac{5z}{6},\frac z3\right),\quad
\left(\frac z3,\frac{z(z+6)}{42}\right),\\[7pt]
&\left(\frac{1+2z}{3},\frac{1+2z}{6}\right),\quad
\left(\frac{5+z}{6},\frac{2+z}{6}\right),\quad 0\leq z\leq1.
\end{aligned}
\end{equation}
On each arc, use the factor $f$ specified in the array \texttt{factors}
of Appendix~\ref{app:mathematica}. The code verifies the factorization
of the substituted polynomial as $f(z)H(z)$ and checks that every
Bernstein coefficient of $H$ is positive. It also checks that $f$, after
division by $z^2$ on the first two arcs, has nonnegative monomial
coefficients and a positive constant term. Hence $f>0$ everywhere
required: only $z=0$ on the first two arcs is excluded, since $Y>0$.
All boundary values are therefore positive, and concavity proves
$P_i>0$ throughout the domain.

This proves the eight inequalities and completes Proposition~\ref{lem:finite}.
The appendix contains all polynomial constructions and sign checks:
267 positive Bernstein coefficients for the slope tests and 376 for
the curvature boundaries, together with 12 and 82 rational lower bounds.
No special-function evaluations enter this finite step.

\appendix
\section{Exact polynomial calculations}\label{app:mathematica}

The following \textsc{Mathematica} code implements the polynomial
calculations of Section~\ref{sec:algebra} in exact rational arithmetic.
It checks polynomial identities, factorizations and coefficient signs;
the arguments showing why these checks suffice are given in the text.
Copy the code without the line numbers into a single input cell and
evaluate; no additional packages are required.

\lstset{style=wlcode}
\begin{lstlisting}[firstnumber=1,
caption={Wolfram Language code for the finite polynomial
calculations.\\},label={lst:mathematica-checks}]
Module[{y, q, z, s, h, t, i, j, k, m, n, check, zeroQ, rationalQ, polynomial,
  bernstein, positiveBernstein, rowBounds, rising, e, u, v, g, lam, ga, va,
  kappa, rr, upper, slopeDen, sp, edge, lower, expectedLower,
  targets, intervals, mapped, order, core, sc = 0, nc, eb,
  expectedEB, eta, jj, tt, dn, independentGA, directionalGA,
  partialY, partialT, blockB, expr, delta0, deltaTail, delta, cp,
  degrees, rowDegrees, minima, bounds, cb = 0, cc = 0,
  fullPositive = 0, fullZero = 0, rules, factors, residualDegrees,
  residual, full, data, names = {0, 4, 5, 6}, seams},
 Catch[
  check[test_, label_] := If[! TrueQ[test],
    Print["FAIL: ", label]; Throw[$Failed]];
  zeroQ[a_] := Cancel[Together[a]] === 0;
  rationalQ[a_] := IntegerQ[a] || Head[a] === Rational;
  polynomial[a_, vars_List] := Module[{p = Expand[Cancel[Together[a]]]},
    check[PolynomialQ[p, vars] &&
      Complement[Variables[p], vars] === {}, "polynomial construction"];
    check[AllTrue[Flatten[CoefficientList[p, vars]], rationalQ],
      "exact rational coefficients"]; p];

  (* Bernstein conversion, including exact reconstruction.
     Coefficients may themselves be polynomials in another variable. *)
  bernstein[p_, x_] := Module[{d = Exponent[p, x], b},
    b = Table[Sum[Coefficient[p, x, j] Binomial[k, j]/Binomial[d, j],
      {j, 0, k}], {k, 0, d}];
    check[Expand[p - Sum[b[[k + 1]] Binomial[d, k]
      x^k (1 - x)^(d - k), {k, 0, d}]] === 0,
      "Bernstein reconstruction"]; b];
  positiveBernstein[p_, label_] := Module[{b = bernstein[p, z]},
    check[AllTrue[b, (rationalQ[#] && TrueQ[# > 0]) &], label];
    Print["PASS: ", label, "; coefficients = ", Length[b]]; Length[b]];
  rowBounds[p_] := Module[{rows},
    rows = CoefficientList[Expand[#], s] & /@ bernstein[p, y];
    check[AllTrue[Flatten[rows], rationalQ], "rational row coefficients"];
    (First[#] + Total[Select[Rest[#], TrueQ[# < 0] &]]) & /@ rows];

  (* Common rational product definitions; t = y + q. *)
  t = y + q;
  rising[x_, n_Integer] := Product[x + j, {j, 0, n - 1}];
  e[n_Integer] := rising[y, n + 1]/rising[t, n + 1];
  u[n_Integer] := y/(y + n + 1);
  v[n_Integer] := u[n] - q e[n];
  g[n_Integer] := e[n] (t + n)/n;
  lam = (y + t + 2)/(y + 3);
  ga = Cancel[(g[2] + lam g[3])/(1 + lam)];
  va = Cancel[(v[2] + lam v[3])/(1 + lam)];
  kappa = 3 - 6 y + 20 q - (8/3) y (1 - y);
  rr = Cancel[ga - kappa va];
  upper = 3 y/7 + 3 y^2/14;

  (* I. sp[[1]], sp[[2]], sp[[3]] are S_4, S_5, S_6;
     edge is S_E. Their displayed denominators are positive products. *)
  slopeDen = {12 rising[t, 5] (y + 4) (y + 5) (3 y + q + 5),
    30 rising[t, 6] (y + 4) (y + 6) (3 y + q + 5),
    6 rising[t, 7] (y + 4) (y + 7) (3 y + q + 5)};
  sp = Table[polynomial[slopeDen[[m - 3]]
    (ga - g[m] - kappa (va - v[m]))/y, {y, q}], {m, 4, 6}];
  edge = polynomial[42 rising[t, 8] (y + 7) (y + 8)
    (g[6] - g[7] - kappa (v[6] - v[7]))/y, {y, q}];
  check[AllTrue[Join[sp, {edge}],
    AllTrue[Flatten[CoefficientList[#, {y, q}]], IntegerQ] &],
    "integer slope polynomials"];
  Do[check[zeroQ[ga - g[m] - kappa (va - v[m]) -
    y sp[[m - 3]]/slopeDen[[m - 3]]],
    {"slope polynomial identity", m}], {m, 4, 6}];
  check[zeroQ[g[6] - g[7] - kappa (v[6] - v[7]) -
    y edge/(42 rising[t, 8] (y + 7) (y + 8))],
    "edge polynomial identity"];
  expectedLower = {
    384 y^7 + 10272 y^6 + 99372 y^5 + 439110 y^4 +
      1267476 y^3 + 2281788 y^2 + 1874448 y + 773280,
    3360 y^8 + 79080 y^7 + 751980 y^6 + 3931980 y^5 +
      (22887195/2) y^4 + 28299570 y^3 + 54157980 y^2 +
      44443140 y + 16900560,
    1440 y^9 + 38520 y^8 + 415554 y^7 + (4561527/2) y^6 +
      (33587967/4) y^5 + (83864667/4) y^4 + 49865688 y^3 +
      97435680 y^2 + 84362436 y + 30170592};
  Do[
    rules = CoefficientRules[Expand[D[sp[[m - 3]], {q, 3}]], {y, q}];
    lower = Expand[Total[Function[a, With[{ij = First[a], c = Last[a]},
      If[ij[[2]] == 0 || c < 0, c y^Total[ij]/2^ij[[2]], 0]]]
      /@ rules]];
    check[Expand[lower - expectedLower[[m - 3]]] === 0 &&
      AllTrue[CoefficientList[lower, y], TrueQ[# > 0] &],
      {"third derivative", m}];
    Print["PASS: third-derivative lower polynomial, m = ", m];
    targets = {
      {sp[[m - 3]] /. q -> 2 y/5, 0, 5/6},
      {sp[[m - 3]] /. q -> upper, 0, 1/3},
      {sp[[m - 3]] /. q -> y - 1/2, 5/6, 1}};
    intervals = If[m == 4, {{1/3, 1}}, {{1/3, 5/6}, {5/6, 1}}];
    targets = Join[targets,
      ({sp[[m - 3]] /. q -> y/2, #[[1]], #[[2]]} & /@ intervals),
      {{-D[sp[[m - 3]], q] /. q -> y/2, 1/3, 1}},
      If[m < 6, {{-D[sp[[m - 3]], {q, 2}] /. q -> upper, 0, 1/3}},
        {{-D[sp[[3]], {q, 2}] /. q -> upper, 0, 1/4},
         {-D[sp[[3]], q] /. q -> upper, 1/4, 1/3}}]];
    check[Length[targets] === m + 2, {"slope target count", m}]; nc = 0;
    Do[
      mapped = polynomial[targets[[k, 1]] /.
        y -> targets[[k, 2]] + (targets[[k, 3]] - targets[[k, 2]]) z,
        {z}];
      order = Min[CoefficientRules[mapped, {z}][[All, 1, 1]]];
      check[order === If[k <= 2, 1, 0] &&
        (order == 0 || targets[[k, 2]] == 0),
        {"excluded endpoint factor", m, k}];
      core = polynomial[mapped/
        ((targets[[k, 3]] - targets[[k, 2]]) z)^order, {z}];
      nc += positiveBernstein[core, {"slope", m, k}],
      {k, Length[targets]}];
    check[nc === {67, 87, 113}[[m - 3]], {"slope coefficient count", m}];
    sc += nc, {m, 4, 6}];
  mapped = polynomial[(edge /. q -> y (2/5 + s/10))/y, {y, s}];
  check[{Exponent[mapped, y], Exponent[mapped, s]} === {11, 9},
    "edge polynomial degrees"];
  eb = rowBounds[mapped];
  expectedEB = {183456, 73135566/275, 2542076726/6875,
    407462558321/825000, 21807956898/34375, 2311736036891/3300000,
    59399319085477/82500000, 11351021013803681/16500000000,
    204106264614529/343750000, 606831935049057/1375000000,
    1665105726353/6250000, 332197901301/2000000};
  check[eb === expectedEB && AllTrue[eb, TrueQ[# > 160000] &],
    "twelve edge bounds"];
  Print["PASS: slope coefficients = ", sc, "; edge bounds = ", Length[eb]];

  (* II. cp[[1]], cp[[2]], cp[[3]], cp[[4]] are P_0, P_4, P_5, P_6.
     These are the curvature polynomials, distinct from the S polynomials. *)
  eta[n_Integer] := Cancel[-v[n]/g[n]];
  jj[n_Integer] := Sum[1/(y + j), {j, 0, n}];
  tt[n_Integer] := Sum[1/(t + j), {j, 0, n - 1}];
  dn[n_Integer] := Cancel[(t + n) (eta[n + 1] - eta[n])];
  independentGA = Cancel[ga /. q -> h - y];
  partialY = Sum[1/(y + j), {j, 0, 3}] +
    2/(5 h + 2 y + 10) - 2/(h + 2 y + 5);
  partialT = 5/(5 h + 2 y + 10) - 1/h - 1/(h + 1) -
    1/(h + 2) - 1/(h + 2 y + 5);
  check[zeroQ[D[independentGA, y]/independentGA - partialY] &&
    zeroQ[D[independentGA, h]/independentGA - partialT],
    "independent partial derivatives of g_A"];
  directionalGA = Cancel[Together[ga ((partialT /. h -> t) +
    rr (partialY /. h -> t))]];
  blockB = Cancel[directionalGA - (3/20) (ga - rr)];
  expr = Join[{Cancel[va (tt[4] - rr jj[4]) - eta[4] blockB]},
    Table[Cancel[va (1 - rr) - blockB dn[n]], {n, 4, 6}]];
  delta0 = 90 (y + 1) (y + 2) (y + 3) (y + 5) (y + 4)^3
    (t + 4) (3 y + q + 5)^3 Product[(t + j)^2, {j, 0, 3}];
  deltaTail[n_Integer] := 180 (y + 1) (y + 2) (y + 3) (y + 4)^2
    (t + n + 1) (3 y + q + 5)^3 Product[(t + j)^2, {j, 0, 3}]
    Product[y + j, {j, 5, n + 2}];
  delta = Expand /@ Join[{delta0}, Table[deltaTail[n], {n, 4, 6}]];
  cp = Table[polynomial[delta[[i]] expr[[i]]/(y (1 - q)), {y, q}],
    {i, 4}];
  degrees = {{19, 16}, {20, 18}, {21, 19}, {22, 20}};
  rowDegrees = {{18, 14}, {19, 16}, {20, 17}, {21, 18}};
  minima = {2571264000, 15396480000, 130170240000, 1221454080000};
  Do[
    data = Last /@ CoefficientRules[delta[[i]], {y, q}];
    check[AllTrue[data, TrueQ[# >= 0] &] && AnyTrue[data, TrueQ[# > 0] &],
      {"positive denominator", names[[i]]}];
    check[AllTrue[Flatten[CoefficientList[cp[[i]], {y, q}]], IntegerQ] &&
      {Exponent[cp[[i]], y], Exponent[cp[[i]], q]} === degrees[[i]],
      {"integer numerator and bidegree", names[[i]]}];
    check[zeroQ[expr[[i]] - y (1 - q) cp[[i]]/delta[[i]]],
      {"curvature identity", names[[i]]}];
    mapped = polynomial[-D[cp[[i]], {q, 2}] /.
      q -> y (2/5 + s/10), {y, s}];
    check[{Exponent[mapped, y], Exponent[mapped, s]} === rowDegrees[[i]],
      {"concavity degrees", names[[i]]}];
    bounds = rowBounds[mapped];
    check[Length[bounds] === 18 + i && AllTrue[bounds, TrueQ[# > 0] &] &&
      First[Ordering[bounds, 1]] === 1 && Min[bounds] === minima[[i]],
      {"concavity bounds", names[[i]]}]; cb += Length[bounds];
    Print["PASS: curvature ", names[[i]], "; bounds = ", Length[bounds],
      "; minimum = ", Min[bounds]], {i, 4}];

  (* Four boundary arcs, mapped to 0 <= z <= 1, in the stated order.
     Only z = 0 on the first two arcs is excluded, since y > 0. *)
  rules = {{y -> 5 z/6, q -> z/3},
    {y -> z/3, q -> z (z + 6)/42},
    {y -> (1 + 2 z)/3, q -> (1 + 2 z)/6},
    {y -> (5 + z)/6, q -> (2 + z)/6}};
  factors = {
    {z^2, z^2 (z + 6)^2, (2 z + 1)^2 (2 z + 7)^2, 1},
    {z^2, z^2 (z + 6)^2, (2 z + 1)^2 (2 z + 7)^3, 1},
    {z^2, z^2 (z + 6)^2, (2 z + 13) (2 z + 1)^2 (2 z + 7)^3, 1},
    {z^2 (z + 6), z^2 (z + 6)^2, (2 z + 1)^2 (2 z + 7)^3, 1}};
  residualDegrees = {{18, 29, 16, 20}, {19, 32, 16, 21},
    {20, 34, 16, 22}, {20, 36, 18, 23}};
  Do[
    mapped = polynomial[cp[[i]] /. rules[[k]], {z}];
    full = bernstein[mapped, z];
    check[AllTrue[full, (rationalQ[#] && TrueQ[# >= 0]) &] &&
      Count[full, 0] === If[k <= 2, 2, 0],
      {"full boundary coefficients", names[[i]], k}];
    fullPositive += Count[TrueQ[# > 0] & /@ full, True];
    fullZero += Count[full, 0];
    residual = polynomial[mapped/factors[[i, k]], {z}];
    check[Expand[mapped - factors[[i, k]] residual] === 0 &&
      Exponent[residual, z] === residualDegrees[[i, k]],
      {"boundary factorization", names[[i]], k}];
    core = polynomial[factors[[i, k]]/If[k <= 2, z^2, 1], {z}];
    check[(core /. z -> 0) > 0 &&
      AllTrue[CoefficientList[core, z], TrueQ[# >= 0] &],
      {"strict elementary factor", names[[i]], k}];
    cc += positiveBernstein[residual, {"curvature boundary", names[[i]], k}],
    {i, 4}, {k, 4}];
  seams = {{y -> 1/3, q -> 1/6}, {y -> 5/6, q -> 1/3},
    {y -> 1, q -> 1/2}};
  check[AllTrue[Flatten[Table[cp[[i]] /. seams[[k]], {i, 4}, {k, 3}]],
    TrueQ[# > 0] &], "strict positivity at junctions"];
  check[{sc, Length[eb], cc, cb, fullPositive, fullZero} ===
    {267, 12, 376, 82, 405, 16}, "final audit counts"];
  Print["PASS: all exact identities and strict inequalities verified."];
  <|"Verified" -> True, "PositiveCoefficients" -> sc + cc,
    "RationalBounds" -> Length[eb] + cb|>
 ]
]
\end{lstlisting}

The function \lstinline[style=wlcode]|bernstein| implements
\eqref{eq:bernstein-conversion}, while
\lstinline[style=wlcode]|rowBounds| computes the lower bounds in
\eqref{eq:row-lower}. The arrays \lstinline[style=wlcode]|sp| and
\lstinline[style=wlcode]|cp| contain the polynomials in
\eqref{eq:slope-polynomials} and \eqref{eq:curvature-polynomials},
respectively; \lstinline[style=wlcode]|factors| specifies the elementary
factors used on the four boundary arcs.
The variables \lstinline[style=wlcode]|y| and
\lstinline[style=wlcode]|q| represent the independent parameters $Y,q$,
with $t=Y+q$. For the partial derivatives of $g_A$, the code uses
\lstinline[style=wlcode]|h| in place of $t$ and differentiates before
substituting \lstinline[style=wlcode]|h = y + q|.

\par\noindent\begin{minipage}{\linewidth}
The expression has local variables and reads no files. A failed test
prints its label and stops evaluation. Successful evaluation verifies
$643$ positive Bernstein coefficients and $94$ rational lower bounds,
and returns
\begin{lstlisting}[style=wlcode,numbers=none]
<|"Verified" -> True, "PositiveCoefficients" -> 643,
  "RationalBounds" -> 94|>
\end{lstlisting}
\end{minipage}
An unnumbered copy of the code is supplied as
\texttt{AskeyElementaryChecks.wl}. All constants must be retained
as integers or fractions, not decimals.

\section*{Acknowledgements}
The \textsc{Mathematica} code in Appendix~\ref{app:mathematica} was developed with the assistance of ChatGPT (OpenAI), based on an earlier version written by the authors that was less efficient but already valid for verification purposes. The authors take responsibility for the final code. The first author acknowledges financial support from the Centre for Mathematics of
the University of Coimbra (CMUC), funded by the Portuguese Foundation for
Science and Technology (FCT), under the projects UID/00324/2025
(\url{https://doi.org/10.54499/UID/00324/2025}) and UID/PRR/00324/2025. The
first author also acknowledges financial support from the FCT under the grant
\url{https://doi.org/10.54499/2022.00143.CEECIND/CP1714/CT0002}.

\makeatletter
\renewcommand{\@biblabel}[1]{\@defaultbiblabelstyle{#1}}
\def\bibsetup{}
\makeatother
\providecommand{\bysame}{\leavevmode\hbox to3em{\hrulefill}\thinspace}
\providecommand{\MR}{\relax\ifhmode\unskip\space\fi MR }
\providecommand{\MRhref}[2]{%
  \href{http://www.ams.org/mathscinet-getitem?mr=#1}{#2}
}
\providecommand{\href}[2]{#2}

\end{document}